\documentclass{article}
\RequirePackage{amsthm}
\theoremstyle{plain}
\newtheorem{theorem}{Theorem}
\theoremstyle{remark}

\theoremstyle{plain}

\theoremstyle{plain}

\theoremstyle{plain}
\newtheorem{lemma}[theorem]{Lemma}
\theoremstyle{plain}
\newtheorem{proposition}[theorem]{Proposition}
\theoremstyle{definition}

\usepackage{mathtools}
\usepackage{enumitem}   
\usepackage{xr}

\usepackage{lmodern}
\RequirePackage[OT1]{fontenc}
\RequirePackage{amssymb}
\RequirePackage[numbers]{natbib}

\newcommand{\beq}{\begin{equation}}
\newcommand{\eeq}{\end{equation}}

\newcommand{\RR}{\mathbb{R}}
\newcommand{\NN}{\mathbb{N}}
\newcommand{\EE}{\mathbb{E}}

\newcommand{\test}{\phi}
\newcommand{\step}{s}

\newcommand{\multiindex}{\alpha}

\newcommand{\hess}{\nabla^2}

\begin{document}








\author{Thomas Bonis\\
	DataShape team, Inria Saclay, Universit\'e Paris-Saclay, Paris, France \\
	thomas.bonis@inria.fr}

	\title{Stein's method for normal approximation in Wasserstein distances with application to the multivariate Central Limit Theorem}

	\maketitle

\begin{abstract}
We use Stein's method to bound the Wasserstein distance of order $2$ between a measure $\nu$ and the Gaussian measure using a stochastic process $(X_t)_{t \geq 0}$ such that $X_t$ is drawn from $\nu$ for any $t > 0$. If the stochastic process $(X_t)_{t \geq 0}$ satisfies an additional exchangeability assumption, we show it can also be used to obtain bounds on Wasserstein distances of any order $p \geq 1$. Using our results, we provide optimal convergence rates for the multi-dimensional Central Limit Theorem in terms of Wasserstein distances of any order $p \geq 2$ under simple moment assumptions. 
\end{abstract}

\section{Introduction}
Consider $n$ independent and, for simplicity, identically distributed random variables $X_1,\dots,X_n$ taking values in $\RR^d$ such that $\EE[X_1] = 0$ and $\EE[X_1 X_1^T] = I_d$. By the Central Limit Theorem, it is well-known that, as $n$ grows to infinity, the law $\nu_n$ of $S_n = \frac{1}{\sqrt{n}} \sum_{i=1}^n X_i$ converges to the $d$-dimensional Gaussian measure $\gamma$. In order to strengthen this result, one can quantify this convergence for a given distance on the space of measures on $\RR^d$. Let us consider the family of Wasserstein distances of order $p \geq 1$, defined between any two measures $\mu$ and $\nu$ with finite moment of order $p$ by 
\[
W_{p}(\nu,\mu)^p =  \inf_{\pi}  \int_{\RR^d \times \RR^d} \|y-x\|^p \pi(dx,dy),
\]  
where $\|\cdot\|$ denotes the Euclidean norm and $\pi$ is a measure on $\RR^d \times \RR^d$ with marginals $\mu$ and $\nu$. 
In the univariate setting, rates of convergence for these distances have been obtained in \citep{rio2009} for $p \in [1,2]$ and in \citep{Bobkov} for $p > 2$. More precisely, for any $p \geq 1$, there exists a constant $C_p > 0$ such that 
\beq
\label{eq:1dim}
W_p(\nu_n, \gamma) \leq \frac{C_p \EE[|X_1|^{p+2}]^{1/p}}{\sqrt{n}}.
\eeq
Furthermore, Theorem 5.1 \citep{rio2009} guarantees this bound to be tight in the general case. In the multivariate setting, convergence rates for the Wasserstein distance of order $2$ have been obtained under the assumption that $\|X_1\| \leq \beta$ with $\beta > 0$, see \citep{Zhai} and \citep{Eldan}, in which case there exists $C > 0$ such that 
\[
W_2(\nu_n, \gamma) \leq C \beta \sqrt{\frac{d \log n}{n}}.
\]
As this result is $\sqrt{\log n}$ short of optimality in the one-dimensional case, it is conjectured in \citep{Zhai} that 
\beq
\label{eq:Zhaiconjecture}
W_2(\nu_n, \gamma) \leq C \beta \sqrt{\frac{d}{n}}
\eeq
and such a bound is known to be matched thanks to Proposition 2 \citep{Zhai}. Let us note that, since $\beta$ is greater than $\sqrt{d}$, this bound scales at least linearly with respect to the dimension which is probably suboptimal in many cases. Indeed, whenever the coordinates of the $X_i$ are i.i.d. random variables with fourth moment equal to $C > 0$, one can use (\ref{eq:1dim}) to obtain the following bound, scaling with $\sqrt{d}$, 
\[
W_2(\nu_n, \gamma) \leq C_2  \sqrt{\frac{C d}{n}}. 
\]
This optimal scaling with respect to the dimension as well as the optimal dependency in $n$ can be obtained whenever the measure of the $X_i$ satisfies a Poincar\'e inequality with constant $C \geq 1$ in which case Theorem~4.1 \citep{Fathi} guarantees that 
\beq
\label{eq:Fathi}
W_2(\nu_n, \gamma) \leq  \sqrt{\frac{(C-1)d}{n}}
\eeq
and similar bounds have also been obtained for Wasserstein distances of any order $p \geq 1$ in \citep{Fathi2}. However, for a measure to satisfy a Poincaré inequality is a strong assumption compared to the simple moment assumption required in the univariate case.

Inequality (\ref{eq:Fathi}) is derived through an approach introduced in \citep{Stein} relying on a object called Stein kernel. Given a probability measure $\nu$ supported on $\RR^d$, a Stein kernel for $\nu$ is a matrix-valued function $\tau_\nu$ such that, for any smooth function $\test$ with compact support,
\[
\int_{\RR^d} -x \cdot \nabla \test(x) + \langle \tau_\nu(x),\nabla^2 \test(x)\rangle_{HS} d\nu(x) = 0,
\]
where $\langle \cdot,\cdot \rangle_{HS}$ is the Hilbert-Schmidt scalar product and $\nabla^2 \test$ denotes the Hessian matrix of $\test$. Since $\nu$ is equal to the Gaussian measure $\gamma$ if and only if $\tau_\nu = I_d$, one can expect $\nu$ to be close to $\gamma$ whenever $\tau_\nu$ is close to $I_d$. This intuition is formalized by the following bound, obtained in Proposition 3.1 \citep{Stein}, 
\[
W_2(\nu, \gamma)^2 \leq \int_{\RR^d} \|\tau_\nu - I_d\|_{HS}^2 d\nu,
\]
where $\| \cdot \|_{HS}$ is the Hilbert-Schmidt norm. 
Furthermore, if $\tau_\nu$ also verifies 
\[
\int_{\RR^d} -x \test(x) + \tau_\nu(x) \nabla \test(x) d\nu(x) = 0,
\]
for any suitable function $\test$, then, by Proposition 3.4 \citep{Stein}, one also has 
\[
W_p(\nu, \gamma)^p \leq C_p \int_{\RR^d} \|\tau_\nu - I_d\|_{p}^p d\nu, 
\]
where $\| \cdot \|_p$ is the Schatten $p$-norm and $C_p >  0$ is a constant depending only on $p$. However, as Stein kernels do not necessarily exist for general measures and can be difficult to compute whenever they do exist, they are not an adequate tool to generalize (\ref{eq:1dim}). 

In this work, we wish to apply the approach developed in \citep{Stein} by replacing Stein kernels with more practical operators $\mathcal{L}_\nu$ satisfying the following property 
\[
\forall \test \in \mathcal{C}^\infty_c, \int_{\RR^d} \mathcal{L}_\nu \test d \nu = 0,
\]
where $\mathcal{C}^\infty_c$ denotes the space of smooth functions with compact support. When an operator $\mathcal{L}_\nu$ verifies this property, in which case we say $\nu$ is invariant under $\mathcal{L}_\nu$, one can expect $\nu$ to be close to $\gamma$ as soon as $\mathcal{L}_\nu$ is similar to the operator $\mathcal{L}_\gamma$ defined by 
\[
\forall \test \in \mathcal{C}^\infty_c, x \in \RR^d, \mathcal{L}_\gamma \test (x) = - x \cdot \nabla \test(x) + \langle I_d, \nabla^2 \test (x) \rangle_{HS}.
\] 
There are many ways to obtain operators $\mathcal{L}_\nu$ under which $\nu$ is invariant; in fact, such operators have been extensively used in Stein's method. For instance, the original approach of Stein \citep{OriginalStein} and its extension to the multidimensional setting  \citep{reinert2009} use pairs of random variables $(X,X')$ both drawn from $\nu$ and such that $(X,X')$ and $(X', X)$ follow the same law. Given such a pair of random variables $(X,X')$, which is called an exchangeable pair, $\nu$ is invariant under the operator $\mathcal{L}_\nu$ defined by 
\beq
\label{eq:exchpairs} 
\forall \test \in \mathcal{C}^\infty_c, x \in \RR^d, \mathcal{L}_\nu \test (x) = \frac{1}{s} \EE[(X' - X) (\test(X') + \test(X)) \mid X = x],
\eeq
where $s >  0$ is a rescaling factor. This operator $\mathcal{L}_\nu$ can then be compared to $\mathcal{L}_\gamma$ using a Taylor expansion. 
In fact, one does not even need an exchangeable pair to apply Stein's method in dimension one. Indeed, as shown by \citep{Rollin}, one can use two random variables $X,X'$, both drawn from $\nu$ but not necessarily forming an exchangeable pair, to construct operators of the form 
\beq
\label{eq:rollinoperator} 
\forall \test \in \mathcal{C}^\infty_c, x \in \RR^d, \mathcal{L}_\nu \test (x) = \frac{1}{s} \EE\left[\int_{0}^{X'} \test(y) dy  - \int_{0}^{X} \test(y) dy \mid X = x\right].
\eeq
Similarly, many other constructs used to apply Stein's method such as zero-bias coupling \citep{zerobias} and size-bias coupling \citep{sizebias} correspond to operators under which $\nu$ is invariant. 

Among these various operators, those defined in (\ref{eq:rollinoperator}) are perhaps the easiest to obtain as they can be constructed from any two random variables $X,X'$ both drawn from the measure $\nu$. However, since there is no notion of primitive functions in higher dimension, such operators are restricted to the univariate setting. Still, in the multidimensional setting, one can use any two random variables $X$ and $X'$ drawn from $\nu$ to define an operator under which $\nu$ is invariant by taking 
\beq 
\label{eq:myoperator}
\forall \test \in \mathcal{C}^\infty_c, \forall x \in \RR^d, \mathcal{L}_\nu \test (x) = \frac{1}{s} \EE[\test(X') - \test(X) \mid X = x]. 
\eeq
Then, given any $\test \in \mathcal{C}^\infty_c$, one can use a Taylor expansion to obtain
\begin{multline*}
s \mathcal{L}_\nu \test (x) = \EE[(X'-X) \cdot \nabla \test(X) + \frac{1}{2} \langle (X'-X) (X'-X)^T, \nabla^2 \test(X) \rangle_{HS} \mid X = x]  \\
+ \mathcal{O}(\EE[\|X'-X\|^3 \mid X = x]). 
\end{multline*}
Thus, one can expect that if 
\begin{itemize}
\item $\frac{\EE[X'-X \mid X]}{s} \approx - X$ ;
\item $\frac{\EE[(X'-X) (X'-X)^T \mid X]}{2s} \approx I_d$ and
\item $\frac{\|X'-X\|^3}{s} \approx 0$
\end{itemize}
then $\mathcal{L}_\nu$ would be similar to $\mathcal{L}_\gamma$ and thus $\nu$ be close to $\gamma$. 
However, one cannot prove such a result by applying the approach of \cite{Stein} to such operators. 
Instead, we use stochastic processes $(X_t)_{t \geq 0}$ such that $X_t$ is drawn from $\nu$ for any $t \geq 0$ and such that $\EE[\|X_t - X_0\|]$ does not grow too fast with respect to $t$ to define a family of operators under which $\nu$ is invariant by taking 
\beq 
\label{eq:operator}
\forall t >  0, \test \in \mathcal{C}^\infty_c, x \in \RR^d, (\mathcal{L}_\nu)_t \test (x) = \frac{1}{s} \EE[\test(X_t) - \test(X_0) \mid X_0 = x].
\eeq
In Theorem~\ref{thm:mainGauss}, we derive bounds for the Wasserstein distance of order $2$ between $\nu$ and the Gaussian measure from such a family of operators. We also provide bounds on Wasserstein distances of any order $p \geq 1$ for one-dimensional normal approximation in Theorem~\ref{thm:WpGauss} and for multidimensional normal approximation in Theorem~\ref{thm:WpGaussexch}. This latter result uses a family of operators of the form (\ref{eq:exchpairs}) and thus requires the pairs $(X_t, X_0)$ and $(X_0, X_t)$ to follow the same law for any $t > 0$.
Let us note that, while we mostly focus on operators defined in (\ref{eq:operator}), proofs of our results can easily be adapted to other operators $\mathcal{L}_\nu$ under which $\nu$ is invariant such as size-bias or zero-bias couplings.  

Our results can be readily applied to obtain rates in the Central Limit Theorem. Indeed, letting $X'_1,\dots,X'_n$ be independent copies of $X_1, \dots,X_n$ and $I$ be a uniform random variable on $\{1,\dots,n\}$, the stochastic process $((S_n)_{t})_{t \geq 0}$ defined by 
\[
\forall t \geq 0, (S_n)_t = S_n + \frac{(X_I' - X_I) 1_{\|X_I\| \vee \|X'_I\| \leq \sqrt{n (e^{2t} - 1)}}}{\sqrt{n}}
 \]
 is such that $((S_n)_t, (S_n)_0)$ and $((S_n)_0, (S_n)_t)$ follow the same law for any $t \geq 0$. Applying our results to this stochastic process, we obtain the following bounds.
\begin{theorem}
\label{thm:CTLmain}
Under the above setting, if $\EE[\|X_1\|^4] < \infty$, then there exists $C < 14$ such that 
\beq
\label{eq:W2mainthm}
W_2(\nu_n, \gamma) \leq \frac{C d^{1/4} \|\EE[X_1 X_1^T \|X_1\|^{2}]\|_{HS}^{1/2}}{\sqrt{n}}.
\eeq
Furthermore, if $\EE[\|X_1\|^{p+2}] <  \infty$ for $p \geq 2$, then there exists $C_p > 0$ depending only on $p$ and such that 
\beq
\label{eq:Wpmainthm}
W_p(\nu_n, \gamma) \leq C_p \frac{d^{1/4} \|\EE[X_1 X_1^T \|X_1\|^{2}]\|_{HS}^{1/2} + \EE[\|X_1\|^{p+2}]^{1/p}}{\sqrt{n}}.
\eeq
\end{theorem}
This result both proves (\ref{eq:Zhaiconjecture}) and generalizes (\ref{eq:1dim}). However, our bound still scales at least linearly with respect to the dimension $d$ and thus fails to generalize (\ref{eq:Fathi}) which can scale with $\sqrt{d}$. Our approach can also be used to obtain more general results, presented in Theorems~\ref{thm:W2CTL} and \ref{thm:WPCTL}, which only require the random variables $X_1,\dots,X_n$ to be independent and provide intermediary rates of convergence under weaker moment assumptions.

The paper is organized as follows. In Section~\ref{sec:notations}, we introduce the notations used in the paper. In Section~\ref{sec:approach}, we present the main arguments we use to apply Stein's method and obtain bounds on the Wasserstein distance of order $2$ in normal approximation. The approach followed to obtain bounds on Wasserstein distances of any order $p$ is then detailed in Section~\ref{sec:Wp}. The computations required to apply our general Wasserstein bounds to obtain rates of convergence in the Central Limit Theorem are presented in Sections~\ref{sec:CTLcomp2} and ~\ref{sec:CTLcompp}. 
Finally, Sections~\ref{sec:technical} and \ref{sec:W2approx} contain technical results and approximation arguments used in the course of this paper. 

\section{Notations and definitions}
\label{sec:notations}
Let $d$ be a positive integer. A $d$-dimensional multi-index $\multiindex$ is a $d$-tuple of non-negative integers
\[
\multiindex = (\multiindex_1, \multiindex_2, \dots, \multiindex_d).
\]
The absolute value of a multi-index $\multiindex$ is given by 
\[
|\multiindex| \coloneqq \sum_{i=1}^d \multiindex_i
\]
and its factorial by 
\[
\multiindex ! \coloneqq \prod_{i=1}^d \multiindex_i !.
\]
For any $x \in \RR^d$ and any multi-index $\multiindex$, let
\[
x^\multiindex \coloneqq \prod_{i=1}^d x_i^{\multiindex_i}.
\]
For any $k \in \NN, x \in \RR^d$, we denote by $x^{\otimes k}$ the family indexed by multi-indices with absolute value $k$ and such that 
\[
\forall |\multiindex| = k, (x^{\otimes k})_\multiindex \coloneqq x^\multiindex. 
\]
In this work, we identify any symmetric matrix $M$ to the family indexed by multi-indices with absolute value $2$ by taking $M_{\multiindex} \coloneqq M_{i,i}$ when $\multiindex(i) = 2$ and $M_{\multiindex} \coloneqq M_{i,j}$ when $\multiindex(i) = \multiindex(j) = 1$.
Let $\langle \cdot, \cdot \rangle$ be the Hilbert Schmidt scalar product defined between
any two families $(x_\multiindex)_{|\multiindex| = k}, (y_\multiindex)_{|\multiindex| = k}$ by
\[
\langle x, y \rangle \coloneqq \sum_{|\multiindex| = k} \frac{k!}{\multiindex !} x_\multiindex y_\multiindex,
\]
and, by extension, 
\[
\|x\|^2 \coloneqq \sum_{ |\multiindex| = k} \frac{k!}{\multiindex !} x_\multiindex^2. 
\]
Let us remark that, for any $k \in \NN, x \in \RR^d$, we have 
\[
\|x^{\otimes k}\| = \|x\|^k. 
\]

Let $\mathcal{C}^k$ be the set of functions from $\RR^d$ to $\RR$ with partial derivatives of order $k \in \NN$ and by $\mathcal{C}^k_c$ the set of such functions with compact support. For any multi-index $\multiindex$ and any $\test \in \mathcal{C}^{|\multiindex|}$, let 
\[
\partial^\multiindex \test = \frac{\partial^{\multiindex_1}}{\partial x_1^{\multiindex_1}} \frac{\partial^{\multiindex_2}}{\partial x_2^{\multiindex_2}} \dots  \frac{\partial^{\multiindex_d}}{\partial x_d^{\multiindex_d}} \test.
\]
Let $\nabla^k \test (x)\in (\RR^{d})^{\otimes k}$ be the $k$-th gradient of $\test$ at $x$ defined by 
\[
\forall |\multiindex| = k, (\nabla^k \test (x))_\multiindex \coloneqq \partial^\multiindex \test.
\]

Let $\gamma$ denoted the $d$-dimensional Gaussian measure and let $\mathcal{L}_\gamma$ be the operator defined by
\[
\forall \test \in \mathcal{C}^k_c, x \in \RR^d, \mathcal{L}_\gamma \test (x) \coloneqq - x \cdot \nabla \test(x) + \langle I_d, \nabla^2 \test(x) \rangle. 
\] 
This operator is the infinitesimal generator of the Ornstein-Uhlenbeck semigroup $(P_t)_{t \geq 0}$ whose reversible measure is $\gamma$; see e.g. \citep{Markov} for a thorough presentation of this semigroup and its properties. 

 \section{Bounds for the Wasserstein distance of order $2$}
\label{sec:approach}

In this Section, we prove the following result. 
\begin{theorem}
\label{thm:mainGauss}
Let $\nu$ be a probability  measure on $\RR^d$ with finite second moment and let $(X_t)_{t \geq 0}$ be a stochastic process such that $X_t$ is drawn from $\nu$ for any $t >  0$. 
Suppose that
\beq
\label{eq:condition}
\forall \epsilon >  0, \exists \xi, M >  0, \forall t \in [\epsilon, \epsilon^{-1}],  \EE\left[e^{\frac{(1+\xi) \|X_t - X_0\|^2}{e^{2t}-1}}\right] \leq M. 
\eeq
Then, for any $s >  0$, 
\[
W_2(\nu,\gamma) \leq \int_0^\infty e^{-t} \EE[S(t)]^{1/2} dt, 
\]
where 
\begin{align*}
S(t) = & \left\|\EE\left[\frac{X_t-X_0}{s} + X_0\mid X_0\right] \right\|^2  \\
& + \frac{1}{e^{2t} - 1}\left\|\EE\left[\frac{(X_t-X_0)^{\otimes 2}}{2s} - I_d \mid X_0\right] \right\|^2 \\
& + \sum_{k>2} \frac{1}{ s^2 k k! (e^{2t} - 1)^{k-1} } \|\EE[(X_t-X_0)^{\otimes k}\mid X_0]\|^2.
\end{align*}
\end{theorem}


Let $\nu$ be a measure on $\RR^d$ and let $(X_t)_{t \geq 0}$ be a stochastic process such that $X_t$ is drawn from $\nu$ for any $t \geq 0$. Let us assume the measure $\nu$ admits a density $h$ with respect to $\gamma$ such that $h = \epsilon + f$ for some constant $\epsilon > 0$ and $f \in \mathcal{C}^\infty_c$ and suppose the stochastic process $\|X_t - X_0\|$ is bounded for any $t > 0$. Let us note that, while such assumptions imply a Stein kernel exists, approximation arguments developed in Section~\ref{sec:W2approx} allow us to lift them in favor of the weaker (\ref{eq:condition}).

For $t > 0$, let $\nu_t$ be the measure with density $P_t h$. Since $\gamma$ is the reversible measure of $P_t$, $\nu_t$ converges to $\gamma$ when $t$ grows to infinity. One can thus bound $W_2(\nu,\gamma)$ by controlling $W_2(\nu, \nu_t)$ for any $t > 0$ and letting $t$ grow. To this end, we use the following inequality, obtained in Lemma 2 \citep{OV},
\[
\frac{d^+}{dt} W_{2}(\nu, \nu_t) \leq \left(\int_{\RR^d} \frac{\|\nabla P_t h(x)\|^2}{P_t h(x)} d \gamma(x) \right)^{1/2} := I(\nu_t)^{1/2}
\]
which yields
\beq
\label{eq:OV}
W_2(\nu, \gamma) \leq \int_0^\infty I(\nu_t)^{1/2} dt.
\eeq
The quantity $I(\nu_t)$ is the Fisher information of the measure $\nu_t$ with respect to $\gamma$. In Proposition 2.4 \citep{Stein}, this quantity is bounded using Stein kernels. In this work, we bound $I(\nu_t)$ using the stochastic process $(X_t)_{t \geq 0}$.
\begin{proposition}
\label{pro:Ibound}
Under the above setting, we have 
\[
I(\nu_t) \leq e^{-2t} \EE[S(t)],
\]
where $S(t)$ is defined in Theorem~\ref{thm:mainGauss}. 
\end{proposition}
As injecting this bound in (\ref{eq:OV})  and using the approximation arguments of Section~\ref{sec:W2approx} concludes the proof of Theorem~\ref{thm:mainGauss}, the remainder of this Section is dedicated to the proof of this Proposition. 

Let $t > 0$ and let $v_t \coloneqq \log P_t h$. By Equation~(2.12) \citep{Stein}, we have 
\[
I(\nu_t) = -\int_{\RR^d}  \mathcal{L}_\gamma P_t v_t d \nu.
\]
Hence, if an operator $\mathcal{L}_\nu$ verifies
\[
\int_{\RR^d} \mathcal{L}_\nu P_t v_t d\nu = 0,
\]
then 
\beq
\label{eq:auxmain}
I(\nu_t) = \int_{E}  (\mathcal{L}_\nu - \mathcal{L}_\gamma) P_t v_t d \nu.
\eeq
Now, let $\step > 0$ and let $\mathcal{L}_\nu$ be the operator such that, for any $\test \in L^1(\nu)$ and any $x \in \RR^d$, 
\[
\mathcal{L}_\nu \test (x) \coloneqq \frac{1}{\step} \EE\left[ \test(X_t) - \test(X_0)  |X_0 = x\right].
\]
Since $X_t$ and $X_0$ are drawn from the same law, integrating this operator with respect to $\nu$ gives 
\[
\int_{\RR^d} \mathcal{L}_\nu \test (x) d\nu(x) = \frac{1}{\step} \EE[\test(X_t) - \test(X_0)] = 0.  
\]
Let us rewrite $\mathcal{L}_\nu$ using a Taylor expansion. 
\begin{lemma}
\label{lem:analytic}
Let $\test$ be a bounded and measurable function and let $t > 0$ and $\multiindex$ be a multi-index. Under the above setting, we have that 
\[ \EE\left[(X_t-X_0)^\multiindex \mid X_0], \partial^\multiindex P_t \test(X_0) \right]
\]
exists and that 
\[
\EE[\test(X_t) - \test(X_0)] = \sum_{|\multiindex| > 0} \EE\left[ \frac{\EE[(X_t-X_0)^{\multiindex} \mid X_0] \partial^\multiindex P_t \test(X_0)}{\multiindex !}  \right]. 
\]
\end{lemma}
We delay the proof of this result to Section~\ref{sec:analproof}. Let $k > 0$ be an integer, after rearranging terms, we have 
\begin{multline}
\label{eq:rearrange}
\sum_{|\multiindex| = k} \frac{\EE[(X_t-X_0)^{\multiindex} \mid X_0] \partial^\multiindex P_t \test(X_0)}{\multiindex !} = \\
\sum_{|\multiindex| = k - 1} \frac{\EE[(X_t - X_0) (X_t-X_0)^{\multiindex} \mid X_0] \cdot \partial^\multiindex \nabla P_t \test(X_0)}{(|\multiindex|+1) \multiindex !}.
\end{multline}
Thus, 
\begin{multline*}
\int_{\RR^d} \mathcal{L}_\nu P_t \test (x) d\nu(x) = \frac{1}{s} \EE\left[\EE[X_t - X_0 \mid X_0] \cdot \nabla P_t v_t (X_0)\right] \\
+ \frac{1}{2s}\EE\left[\left \langle \EE[(X_t - X_0)^{\otimes 2} \mid X_0] , \hess P_t v_t (X_0) \right \rangle \right] \\ 
+ \sum_{|\multiindex| > 1} \frac{\EE\left[\EE[(X_t - X_0) (X_t-X_0)^{\multiindex} \mid X_0] \cdot \partial^\multiindex \nabla P_t \test(X_0) \right]}{s (|\multiindex|+1) \multiindex !}.
\end{multline*}
Then, by (\ref{eq:auxmain}), 
\begin{multline}
\label{eq:Ibound}
I(\nu_t) = \EE\left[ \left(\frac{\EE[X_t - X_0 \mid X_0]}{s} + X_0 \right) \cdot \nabla P_t v_t (X_0) \right] \\
+ \EE \left[\left \langle \frac{\EE[(X_t - X_0)^{\otimes 2} \mid X_0]}{2s} - I_d, \hess P_t v_t (X_0) \right \rangle \right] \\ 
 + \sum_{|\multiindex| > 1}  \frac{\EE\left[\EE[(X_t - X_0) (X_t-X_0)^\multiindex  \mid X_0] \cdot \partial^\multiindex \nabla P_t v_t (X_0) \right]}{(|\multiindex|+1) \multiindex ! s}.
\end{multline}
Let $\test$ be a bounded and measurable function. By Equation (2.7.3) \citep{Markov}, 
\beq
\label{eq:OU}
P_t \test (x) = \int_{\RR^d} \test(xe^{-t} + \sqrt{1-e^{-2t}}y) d\gamma(y).
\eeq
In particular if $\test$ is a function such that $\|\nabla \test\|$ is bounded, we have $\nabla P_t \test = e^{-t} P_t \nabla \test$. 
For any multi-index $\multiindex$, let $H_\multiindex$ be the multivariate Hermite polynomial of index $\multiindex$, defined for any $x \in \RR^d$ by 
\[
H_\multiindex(x) \coloneqq (-1)^k e^{\frac{\|x\|^2}{2}} \partial^\multiindex e^{-\frac{\|x\|^2}{2}}.
\]
Let $\test \in \mathcal{C}^\infty$ be a bounded function. For any multi-index $\multiindex$, starting with (\ref{eq:OU}) and integrating $|\multiindex|$ times with respect to the Gaussian measure, we obtain 
\beq
\label{eq:ippgauss}
\partial^\multiindex P_t \test(x) = \frac{1}{(e^{2t}-1)^{|\multiindex|/2}}  \int_{\RR^d} H_\multiindex (y) \test(xe^{-t} + \sqrt{1-e^{-2t}}y) d\gamma(y).
\eeq
Since Hermite polynomials form an orthogonal basis of $L^2(\gamma)$ with norms
\[
\forall \multiindex, \|H_\multiindex\|^2_{\gamma} \coloneqq \int_{\RR^d} H_\multiindex^2 (y) d\gamma(y) = \multiindex !,
\]
applying (\ref{eq:ippgauss}) to the vector field $\nabla v_t$ yields, for any $x \in \RR^d$ and any multi-index $\multiindex$,
\begin{align*}
\frac{ (e^{2t} - 1)^{|\multiindex|}}{e^{-2t} \multiindex !} \|\partial^\multiindex \nabla P_t v_t (x)\|^2
& = \frac{ (e^{2t} - 1)^{|\multiindex|}}{ \multiindex !} \|\partial^\multiindex P_t \nabla v_t (x)\|^2\\ 
& = \left \| \int_{\RR^d} \frac{H_\multiindex(y)}{\|H_\multiindex\|_{\gamma}}  \nabla v_t (xe^{-t} + \sqrt{1-e^{-2t}}y)  d\gamma(y) \right \|^2 \\
\end{align*}
Therefore, 
\beq
\label{eq:sommeipp}
\sum_{\multiindex} \frac{(e^{2t} - 1)^{|\multiindex|}}{e^{-2t} \multiindex!}   \left\|\partial^\multiindex \nabla  P_t v_t (x)\right\|^2  = P_t \|\nabla v_t(x)\|^2 .
\eeq
Now, let 
\begin{align*}
S(t) \coloneqq &  \left\|\EE\left[\frac{X_t-X_0}{s} + X_0\mid X_0\right] \right\|^2  \\
& + \frac{1}{e^{2t} - 1} \left\|\EE\left[\frac{(X_t-X_0)^{\otimes 2} }{2s} - I_d\mid X_0\right] \right\|^2 \\
& + \sum_{|\multiindex| > 1} \frac{\|\EE[(X_t-X_0) (X_t-X_0)^\multiindex \mid X_0]\|^2}{ (s (|\multiindex|+1))^2 \multiindex ! (e^{2t} - 1)^{|\multiindex|} } .
\end{align*}
Applying Cauchy-Schwarz inequality on (\ref{eq:Ibound}) and using (\ref{eq:sommeipp}), we obtain 
\[
I(\nu_t) \leq e^{-t} \EE[S(t)]^{1/2} \EE[P_t \|\nabla v_t(X_0)\|^2]^{1/2}.
\]
Then, since $v_t = \log(P_t h)$, 
\begin{align*}
I(\nu_t) & \leq e^{-t}\EE[S(t)]^{1/2}\left(\int_{\RR^d} \frac{\|\nabla P_t h\|^2}{(P_t h)^2}  \, d\nu_t\right)^{1/2} \\
& \leq e^{-t}\EE[S(t)]^{1/2}\left(\int_{\RR^d} \frac{\|\nabla P_t h\|^2}{P_t h}  \, d\gamma\right)^{1/2} \\
& \leq e^{-t}\EE[S(t)]^{1/2}I(\nu_t)^{1/2}.
\end{align*}
Finally, since $I(\nu_t)$ is finite,
\[
I(\nu_t)^{1/2} \leq e^{-t}\EE[S(t)]^{1/2},
\]
and rearranging terms in $S(t)$ using (\ref{eq:rearrange}) 
concludes the proof of Proposition~\ref{pro:Ibound}. 
\section{Gaussian measure and Wasserstein distances of any order}
\label{sec:Wp}
Let $p \geq 1$ and let $\nu$ be a measure on $\RR^d$. 
Let us assume the measure $\nu$ admits a density $h$ with respect to $\gamma$ such that $h = \epsilon +f$ with $\epsilon > 0$ and $f \in \mathcal{C}^\infty_c$. 

In order to bound the $W_p$ distance between $\nu$ and the $d$-dimensional Gaussian measure $\gamma$, it is possible to use Stein kernels to obtain a version of the score function $\nabla v_t \coloneqq \nabla \log P_t h$ \citep{Stein}. Indeed, by Section 3 \citep{WangOV}, this score function can be used to bound the Wasserstein distances between $\nu$ and $\gamma$ as 
\[
\frac{d+}{dt} W_p(\nu, \nu_t) \leq \left( \int_{\RR^d} \|  \nabla v_t \|^p d\nu_t \right)^{1/p},
\]
leading to
\beq
\label{eq:OVp}
W_p(\nu,\gamma) \leq \int_0^\infty \left( \int_{\RR^d} \|  \nabla v_t \|^p d\nu_t \right)^{1/p} dt.
\eeq
Let us provide a version of $v_t$. Let $Z$ be a Gaussian random variable, $X_0$ be a random variable drawn form $\nu$ and let $F_t \coloneqq e^{-t} X_0 + \sqrt{1-e^{-2t}} Z$.
\begin{lemma}
\label{lem:score}
Let $t > 0$. Then, under the above notations,
\[
\rho_t \coloneqq \EE\left[e^{-t}X_0 - \frac{e^{-2t}}{\sqrt{1-e^{-2t}}} Z\mid F_t\right]
\] 
is a version of $\nabla v_t (F_t)$. 
\end{lemma}
\begin{proof}
Let $t > 0$.
Integrating by parts with respect to $\gamma$, we have, for any $\test \in \mathcal{C}^\infty_c$, 
\begin{align*}
\int_{\RR^d} \nabla \test (x) d \nu_t (x) &= \int_{\RR^d} \nabla \test (x) P_t h (x) d \gamma (x)  \\
 & = \int_{\RR^d} \nabla  (\test P_t h) (x) - \test (x) \nabla P_t h (x) d \gamma (x) \\
 & = \int_{\RR^d} x  \test (x) P_t h (x) - \test (x) \nabla P_t h (x) d\gamma (x) \\
 & = \int_{\RR^d} P_t h(x)  \test (x) \left(x - \frac{\nabla P_t h (x)}{P_t h(x)}\right) d\gamma (x) \\
 & = \int_{\RR^d}  \test (x) \left(x - \frac{\nabla P_t h (x)}{P_t h(x)}\right) d\nu_t (x).
\end{align*}
Thus, 
\beq
\label{eq:scorecarac}
\forall \test \in \mathcal{C}^\infty_c, \int_{\RR^d} \nabla \test (x) d \nu_t (x) = \int_{\RR^d} \test (x) ( x - \nabla v_t (x)) d \nu_t (x).
\eeq
In fact, this property completely characterizes $\nabla v_t$: if another vector field $\xi : \RR^d\rightarrow \RR^d$ satisfies
\[
\forall \test \in \mathcal{C}^\infty_c,  \int_{\RR^d} \nabla \test (x) d \nu_t (x) =  \int_{\RR^d} \test (x) ( x - \xi (x)) d \nu_t (x),
\]
then  
\[
\forall \test \in \mathcal{C}^\infty_c,  \int_{\RR^d} \test (x) (\nabla v_t (x) - \xi(x)) d\nu_t = 0,
\]
implying that $\xi = \nabla v_t$ almost everywhere with respect to the measure $\nu_t$.  

Now, let $\test \in \mathcal{C}^\infty_c$. Integrating by parts with respect to the Gaussian measure, we have 
\begin{align*}
\EE[\test(F_t) (F_t - \rho_t)] & = \EE\left[\test(F_t) \left(F_t - e^{-t}X_0 + \frac{e^{-2t}}{\sqrt{1-e^{-2t}}} Z\right)\right] \\
& = \EE\left[\frac{1}{\sqrt{1-e^{-2t}}}\test(F_t) Z\right] \\
& = \EE[\nabla \test(F_t)],
\end{align*}
implying that $\rho_t$ it is a version of $\nabla v_t$. 
\end{proof}

Bounding $W_p(\nu,\gamma)$ can thus be achieved by estimating $\EE[\|\rho_t\|^p]$, where $\rho_t$ is defined in Lemma~\ref{lem:score}. 
To this end, suppose there exists a quantity $\tau_t$ such that $\EE[\tau_t \mid F_t] = 0$ almost surely. Then, 
\begin{align*}
\EE[\|\rho_t\|^p] & = \EE\left[\left\|\EE\left[e^{-t}X_0 - \frac{e^{-2t}}{\sqrt{1-e^{-2t}}} Z\mid F_t \right]\right\|^p\right] \\
& = \EE\left[\left\|\EE\left[\tau_t + e^{-t}X_0 - \frac{e^{-2t}}{\sqrt{1-e^{-2t}}} Z\mid F_t \right]\right\|^p\right]
\end{align*}
and, by Jensen's inequality, 
\beq
\label{eq:taut}
\EE[\|\rho_t\|^p]  \leq \EE\left[\left\|\tau_t + e^{-t}X_0 - \frac{e^{-2t}}{\sqrt{1-e^{-2t}}} Z\right\|^p\right].
\eeq
Therefore, if such a quantity $\tau_t$ is close to $e^{-t}X_0 - \frac{e^{-2t}}{\sqrt{1-e^{-2t}}} Z$ then $\EE[\|\rho_t\|^p]$ is small and, by (\ref{eq:OVp}), so is $W_p(\nu, \gamma)$. Before showing how to compute such quantities in the following Sections, let us state the following result, proved in Section~\ref{sec:hypercon}.
\begin{lemma}
\label{lem:hypercon}
Let $Z$ be a normal random variable and let $(M_\multiindex)_{\multiindex \in \NN^d} \in \RR^d$. Then, 
\[
\EE[\|\sum_{\multiindex} M_{\multiindex} H_\multiindex(Z) \|^p]^{2/p} \leq \sum_{\multiindex} \max(1, p-1)^{|\multiindex|} \multiindex ! \|M_{\multiindex}\|^2.
\]
\end{lemma}

\subsection{One-dimensional case}
\label{sec:WpGauss1}

In this Section, we bound the $W_p$ distance between $\nu$ and $\gamma$ in the case $d = 1$ and obtain the following result. 

\begin{theorem}
\label{thm:WpGauss}
Let $p \geq 1$ and 
let $\nu$ be a probability measure on $\RR$ with finite moment of order $p$. Let $(X_t)_{t \geq 0}$ be a stochastic process such that $X_t$ is drawn from $\nu$ for any $t >  0$. 
Suppose that
\beq
\label{eq:condition2}
\forall \epsilon >  0, \exists \xi, M >  0, \forall t \in [\epsilon, \epsilon^{-1}],  \EE\left[e^{\frac{ p (1 + \xi) \max(1, p-1) |X_t - X_0|^2}{2(e^{2t}-1)}}\right] \leq M. 
\eeq
Then, for any $s >  0$,
\[
W_p(\nu,\gamma) \leq \int_0^\infty e^{-t} \EE[S_p(t)^{p/2}]^{1/p} dt,
\]
where 
\begin{align*}
S_p(t) =  &  \EE\left[\frac{X_t-X_0}{s} +X_0 \mid X_0\right]^2 \\
& + \frac{\max(1, p-1)}{e^{2t} -1}  \EE\left[\frac{(X_t-X_0)^{2}}{2 s}  - 1 \mid X_0 \right]^2 \\
& +  \sum_{k>2} \frac{\max(1, p-1)^{k-1}}{s^2 k k! (e^{2t}-1)^{k-1} }  \EE[(X_t-X_0)^{k}\mid X_0]^2.
\end{align*}
\end{theorem}

Let $s > 0$ and let $(X_t)_{t \geq 0}$ be a stochastic process such that for any $t \geq 0$, $X_t$ is drawn from $\nu$ and $\|X_t - X_0\|$ is bounded. Again, thanks to approximation arguments developed in Section~\ref{sec:W2approx}, this assumption as well as the assumptions made on the smoothness of the measure $\nu$ can be lifted in favor of the more general (\ref{eq:condition2}). 
For now, let us start by using $(X_t)_{t \geq 0}$ to obtain a quantity $\tau_t$ such that $\EE[\tau_t \mid F_t] = 0$. 
\begin{lemma}
\label{lem:tau1}
Let $s,t > 0$. Letting
\[
\tau_t \coloneqq \sum_{k > 0} \frac{e^{-kt}}{s k! \sqrt{1-e^{-2t}}^{k-1}} \EE[(X_t-X_0)^k  \mid X_0] H_{k-1}(Z),
\] 
where $H_k$ is the one-dimensional $k$-th Hermite polynomial, we have 
\[
\EE[\tau_t \mid F_t] = 0. 
\]
\end{lemma}

\begin{proof}
Let $\test \in \mathcal{C}^\infty_c$. For any $k \in \NN$, we denote by $\test^{(k)}$ the $k$-th derivative of $\test$. 
Let $k \in \NN$. Since $X_0$ and $Z$ are independent, applying (\ref{eq:ippgauss}) yields 
\[
\EE[H_k(Z) \test(F_t)] = \EE[H_k(Z) \test(e^{-t}X_0 + \sqrt{1 - e^{-2t}}Z)] = (1 - e^{-2t})^{k/2} \EE[\test^{(k)} (F_t)].
\]
Thus, 
\begin{align*}
\EE[\EE[\tau_t \mid F_t] \test(F_t)] & = \EE[\tau_t \test(F_t)]\\
& =  \frac{1}{s} \EE\left[\sum_{k > 0} \frac{e^{-kt}}{k!} (X_t-X_0)^k \test^{(k-1)}(F_t)\right].
\end{align*}
Now, let $\Phi$ be a primitive function of $\test$. By Lemma~\ref{lem:analytic}, the function $x \rightarrow \EE[\Phi(F_t)\mid X_0=x] = P_t \Phi (x)$ satisfies 
\begin{align*}
\EE[P_t \Phi(X_t) -P_t \Phi(X_0)]  & =\EE\left[\sum_{k > 0} \frac{e^{-kt}}{k!}  (X_t-X_0)^k (P_t \test)^{(k-1)}(X_0)\right]\\
& = \EE\left[\sum_{k > 0} \frac{e^{-kt}}{k!} (X_t-X_0)^k \test^{(k-1)}(F_t)\right] \\
& = s \EE[\EE[\tau_t \mid F_t] \test(F_t)].
\end{align*}
Then, since $X_t$ and $X_0$ are both drawn from $\nu$,
\[
\EE[\EE[\tau_t \mid F_t] \test(F_t)] = \frac{1}{s}\EE[P_t \Phi(X_t) -P_t \Phi(X_0)]  = 0,
\]
implying that $\EE[\tau_t \mid F_t] = 0$ almost surely. 
\end{proof}
Returning to the proof of Theorem~\ref{thm:WpGauss}, letting $s,t > 0$ and using Lemma~\ref{lem:tau1} along with 
Lemma~\ref{lem:score} and Jensen's inequality, we obtain 
\begin{align*}
\EE[|\rho_t|^p] & = \EE\left[\left|\EE\left[e^{-t} X_0 + \frac{e^{-2t}}{\sqrt{1-e^{-2t}}} Z + \tau_t \mid F_t\right]\right|^p\right] \\
& \leq \EE\left[\left|e^{-t} X_0 + \frac{e^{-2t}}{\sqrt{1-e^{-2t}}} Z + \tau_t\right|^p\right] \\
& \leq \EE_{X_0}\left[\EE_{Z}\left[\left|e^{-t} X_0 + \frac{e^{-2t}}{\sqrt{1-e^{-2t}}} Z + \tau_t\right|^p\right]\right].
\end{align*}
Then, by Lemma~\ref{lem:hypercon},
\[
\EE[|\rho_t|^p]^{1/p} \leq  \EE[S_p(t)^{p/2}]^{1/p},
\]
where
\begin{align*}
S_p(t) \coloneqq & \EE\left[\frac{X_t-X_0}{s} +X_0 \mid X_0\right]^2 \\
& + \frac{ \max(1,p-1)}{ e^{2t} -1}  \EE\left[\frac{(X_t-X_0)^{2}}{2 s}  - 1 \mid X_0 \right]^2 \\
& +  \sum_{k>2} \frac{\max(1,p-1)^{k-1}}{s^2 k k!  (e^{2t}-1)^{k-1} }  \EE\left[(X_t-X_0)^k \mid X_0\right]^2.
\end{align*}
Finally, by (\ref{eq:OVp}), 
\[
W_p(\nu,\gamma) \leq \int_0^\infty \EE[|\rho_t|^p]^{1/p} dt \leq \int_0^\infty e^{-t} \EE[S_p(t)^{p/2}]^{1/p} dt,
\]
and using approximation arguments concludes the proof of Theorem~\ref{thm:WpGauss}. 
\subsection{Multi-dimensional case}
\label{sec:WpGaussexch}
Unfortunately, it is not possible to use a multi-dimensional generalization of the random vector $\tau_t$ defined in Lemma~\ref{lem:tau1} as we would only be able to show that
\[
\forall \test \in \mathcal{C}^\infty_c, \EE[\EE[\tau_t \mid F_t] \cdot \nabla \test(F_t)] = 0,
\]
which is not sufficient to assert that $\EE[\tau_t \mid F_t] = 0$. Instead, one can add an exchangeability assumption on the stochastic process $(X_t)_{t \geq 0}$ to obtain the following result. 

\begin{theorem}
\label{thm:WpGaussexch}
Let $p \geq 1$ and 
let $\nu$ be a probability measure on $\RR^d$ with finite moment of order $p$. Let $(X_t)_{t \geq 0}$ be a stochastic process such that $X_0$ is drawn from $\nu$ and such that the pairs $(X_0,X_t)$ and $(X_t, X_0)$ follow the same law for any $t >  0$. 
Suppose that, for any $\epsilon >  0$,
\beq
\label{eq:condition3}
\exists \xi, M >  0, \forall t \in [\epsilon, \epsilon^{-1}],  \EE\left[\|X_t - X_0\|^{p(1+\xi)} e^{\frac{ p(1+\xi) \max(1, p-1) \|X_t - X_0\|^2}{2(e^{2t}-1)}}\right] \leq M. 
\eeq
Then, for any $s >  0$, 
\[
W_p(\nu,\gamma) \leq \int_0^\infty e^{-t} \EE[S_p(t)^{p/2}]^{1/p} dt,
\]
where
\begin{align*}
S_p(t)  & = \left\|\EE\left[\frac{X_t-X_0}{s}+X_0 \mid X_0\right]\right\|^2 \\
& +  \frac{\max(1, p-1)}{e^{2t}-1}  \left\|\EE\left[\frac{(X_t-X_0)^{\otimes 2}}{2 s} - I_d \mid X_0 \right]\right\|^2 \\
& +  \sum_{k>2} \frac{\max(1, p-1)^{k-1}}{4 s^2 (k-1)! (e^{2t}-1)^{k-1}} \left\|\EE[(X_t-X_0)^{\otimes k} \mid X_0] \right\|^2.
\end{align*}
\end{theorem}

Let $s > 0$ and let $(X_t)_{t \geq 0}$ be a stochastic process such that, for any $t \geq 0$, $(X_t, X_0)$ and $(X_0, X_t)$ follow the same law and $\|X_t - X_0\|$ is bounded. Again, this last assumption as well as our previous smoothness assumptions on the measure $\nu$ can be replaced by (\ref{eq:condition3}) thanks to approximation arguments derived in Section~\ref{sec:W2approx}. Let us start by using the stochastic process $(X_t)_{t \geq 0}$ to define a quantity $\tau_t$ such that $\EE[\tau_t \mid F_t] = 0$. 
\begin{lemma}
\label{lem:tau2}
Let $s,t > 0$.
The quantity 
\[
\tau_t \coloneqq \left[\frac{e^{-t}}{2s} (X_t-X_0) \left( 1 + \sum_{\multiindex} \frac{(X_t-X_0)^{\multiindex} H_\multiindex(Z)}{\multiindex! (e^{2t}-1)^{|\multiindex|/2}} \right) \mid X_0, Z \right]
\]
satisfies 
\[
\EE[\tau_t \mid F_t] = 0. 
\]
\end{lemma}

\begin{proof}
Let $\test \in \mathcal{C}^\infty_c$. We have
\[
\EE[\EE[\tau_t \mid F_t] \test(F_t)] = \EE[\tau_t \test(F_t)].
\]
Hence, by (\ref{eq:ippgauss}), 
\begin{multline*}
\EE[\EE[\tau_t \mid F_t] \test(F_t)] = \\
\frac{e^{-t}}{2s}  \EE\left[(X_t-X_0) \left(\test(F_t) + \sum_{\multiindex} \frac{e^{-|\multiindex|t}}{\multiindex !} (X_t-X_0)^\multiindex \partial^\multiindex \test(F_t) \right)\right].
\end{multline*}
Let $F'_t \coloneqq e^{-t}X_t + \sqrt{1 - e^{-2t}} Z$. 
By Lemma~\ref{lem:analytic}, we have 
\[
\EE[\EE[\tau_t \mid F_t] \test(F_t)]  = \frac{e^{-t}}{2s} \EE\left[(X_t-X_0) \left(\test(F_t) + \test(F'_t)  \right)\right].
\]
Then, since the pairs $(X_0,X_t)$ and $(X_t,X_0)$ follow the same law, 
\[
\EE\left[ (X_t-X_0)(\test(F_t) + \test(F'_t))\right] = 0
\]
 and thus 
$\EE[\tau_t \mid F_t] = 0$.
\end{proof}

Returning to the proof of Theorem~\ref{thm:WpGaussexch} and using Lemma~\ref{lem:tau2} along with 
Lemma~\ref{lem:score} and Jensen's inequality, we obtain 
\[
\EE[\|\rho_t\|^p]^{1/p} \leq \EE\left[\left\|e^{-t} X_0 + \frac{e^{-2t}}{\sqrt{1-e^{-2t}}} Z + \tau_t\right\|^p\right]^{1/p}.
\] 
Thus, by Lemma~\ref{lem:hypercon}, 
\[
\EE[\|\rho_t\|^p]^{1/p} \leq e^{-t} \EE[\|S_p(t)\|^{p/2}]^{1/p}, 
\]
where
\begin{align*}
S_p(t) \coloneqq & \left\|\EE\left[\frac{X_t-X_0}{s}+X_0 \mid X_0\right]\right\|^2 \\
& + \frac{\max(1, p-1)}{e^{2t}-1}  \left\|\EE\left[\frac{(X_t-X_0)^{\otimes 2}}{2 s} - I_d \mid X_0 \right]\right\|^2 \\
& + \sum_{|\alpha| > 1} \frac{\max(1, p-1)^{|\alpha|} }{4 s^2 \alpha! (e^{2t}-1)^{|\alpha|}} \|\EE[(X_t - X_0)(X_t-X_0)^{\alpha} \mid X_0]\|^2.
\end{align*} 
Then, injecting this bound in (\ref{eq:OVp}) yields 
\[
W_p(\nu,\gamma) \leq \int_0^\infty e^{-t} \EE[S_p(t)^{p/2}]^{1/p} dt.
\]
Finally, rearranging terms in $S_p(t)$ using (\ref{eq:rearrange}) and using approximation arguments concludes the proof of Theorem~\ref{thm:WpGaussexch}. 

\section{Central Limit Theorem for the $W_2$ distance}
\label{sec:CTLcomp2}
Let $m \in (0,2]$ and $X_1,\dots,X_n$ be independent random variables taking values in $\RR^d$ and such that
\begin{itemize}
\item $\forall i \in \{1,\dots,n\}, \EE[X_i] = 0$; 
\item $\sum_{i=1}^n \EE[X_i^{\otimes 2}] = n I_d$ and
\item $\sum_{i=1}^n \EE[\|X_i\|^{2+m}] < \infty$.
\end{itemize} 
It is known that the measure $\nu_n$ of the random variable $S_n \coloneqq n^{-1/2} \sum_{i=1}^n X_i$ converges to the Gaussian measure $\gamma$. The remainder of this Section is dedicated to quantifying this convergence for the Wasserstein distance of order $2$ in order to obtain the following result. 
\begin{theorem}
\label{thm:W2CTL}
Under the above setting, taking 
\begin{align*}
& C = 8 + \sum_{k > 0} \frac{4^k}{k (k)!}, \\
& M(0)^2 = \frac{1}{n}\sum_{i=1}^n \sum_{k >0} \frac{16^k}{2k (2k)!} \|\EE[X_i^{\otimes 2}]\|^2, \\
& \forall l \in [0,2], M(l)^2 = M(0)^2 + \frac{1}{n}\sum_{i=1}^n 8 \|\EE[X_i^{\otimes 2}]\|^2 \EE[\|X_i\|^l]^2 + 4 \|\EE[X_i^{\otimes 2} \|X_i\|^l]\|^2 ,
\end{align*}
we have, for any $n > 4$,
\begin{multline*}
W_2(\nu_n,\gamma) \leq \frac{ \left(C \sum_{i=1}^n \EE[\|X_i\|^{2+m}]\right)^{1/2}}{n^{(2+m)/4}} + \frac{\left(2 \sum_{i=1}^n \|\EE[X_i^{\otimes 2}]\|^2\right)^{1/2}}{n} + \\
 \begin{cases}
\frac{\sqrt{2} M(0)}{n^{1/2}} + \frac{M(m) \log(n)}{2 n^{m/2}} \text{ if $m < 1$} \\
\frac{\sqrt{2} M(0)}{n^{1/2}} + \frac{M(1) \log(n)}{2 n^{1/2}} \text{ if $ 1 \leq m < 2$}\\
2 \left(\frac{M(0) M(2) }{n}\right)^{1/2} \text{ if $m = 2$}
\end{cases}.
\end{multline*}
\end{theorem}

Let $X'_1,\dots,X'_n$ be independent copies of the variables $X_1, \dots, X_n$. For any $t > 0$, let $\Delta(t) \coloneqq e^{2t} - 1$ and
\[
(S_{n})_t \coloneqq S_n + n^{-1/2}(X'_I - X_I)1_{\|X'_I\| \vee \|X_I\| \leq \sqrt{n \Delta(t)}},
\]
where $I$ is a uniform random variable taking values in $\{1,\dots,n\}$ and $\|X'_I\| \vee \|X_I\|$ denotes the maximum between $\|X'_I\|$ and $\|X_I\|$. 

For any $t \geq 0$, $(S_{n})_t$ is drawn from the same measure as $S_n$ and $\|(S_{n})_t - S_n\| \leq 2 \sqrt{n \Delta(t)}$. 
Thus, we can apply Theorem~\ref{thm:mainGauss} to the measure $\nu_n$ of $S_n$ using the stochastic process $((S_n)_t)_{t \geq 0}$ with $s = \frac{1}{n}$ to obtain 
\[
W_2(\nu_n, \gamma) \leq \int_0^\infty e^{-t} \EE[S(t)]^{1/2} dt, 
\]
where 
\begin{align*}
S(t) \coloneqq & \EE\left[\sqrt{n} (X'_I - X_I)1_{\|X'_I\| \vee \|X_I\| \leq \sqrt{n \Delta(t)}} + S_n \mid S_n \right]^2 \\
& +  \frac{1}{\Delta(t)}  \left\|\EE\left[\frac{(X'_I - X_I)^{\otimes 2}}{2} 1_{\|X'_I\| \vee \|X_I\| \leq \sqrt{n \Delta(t)}} - I_d \mid S_n \right]\right\|^2 \\
& +  \sum_{k > 2} \frac{\|\EE[(X'_I - X_I)^{\otimes k} 1_{\|X'_I\| \vee \|X_I\| \leq \sqrt{n \Delta(t)}} \mid S_n]\|^2}{n^{k - 2} k k ! \Delta(t)^{k - 1}}.
\end{align*}
Let us bound $S(t)$ for $t > 0$. First, since $I$ and $S_n$ are independent, 
\begin{align*}
S(t) =  & \EE\left[\frac{1}{\sqrt{n}} \sum_{i=1}^n \left((X'_i - X_i)1_{\|X'_i\| \vee \|X_i\| \leq \sqrt{n \Delta(t)}} + X_i \right) \mid S_n \right]^2 \\
& + \frac{1}{\Delta(t)}  \left\|\EE\left[\sum_{i=1}^n \left( \frac{(X'_i - X_i)^{\otimes 2}}{2n} 1_{\|X'_i\| \vee \|X_i\| \leq \sqrt{n \Delta(t)}} \right) - I_d \mid S_n \right]\right\|^2 \\
& + \sum_{k > 2} \frac{\left\|\EE\left[\sum_{i=1}^n (X'_i - X_i)^{\otimes k} 1_{\|X'_i\| \vee \|X_i\| \leq \sqrt{n \Delta(t)}} \mid S_n\right]\right\|^2}{n^{k} k k ! \Delta(t)^{k - 1}} .
\end{align*}
Then, since $\EE[X'_i] = 0$ and since $X'_i$ and $S_n$ are independent, 
\begin{multline*}
\EE\left[\sum_{i=1}^n (X'_i - X_i)1_{\|X'_i\| \vee \|X_i\| \leq \sqrt{n \Delta(t)}} + X_i \mid S_n \right] = \\
\EE\left[\sum_{i=1}^n (X_i - X'_i)1_{\|X'_i\| \vee \|X_i\| \geq \sqrt{n \Delta(t)}} \mid S_n \right]
\end{multline*}
and, since $\sum_{i=1}^n \EE[X_i^{\otimes 2}] = n I_d$,
\begin{align*}
S(t) =  & \frac{1}{n} \EE\left[\sum_{i=1}^n (X'_i - X_i)1_{\|X'_i\| \vee \|X_i\| \geq \sqrt{n \Delta(t)}} \mid S_n \right]^2 \\
& +  \frac{\left\|\EE\left[\frac{1}{n} \sum_{i=1}^n \left( (X'_i - X_i)^{\otimes 2} 1_{\|X'_i\| \vee \|X_i\| \leq \sqrt{n \Delta(t)}} - 2 \EE[X_i^{\otimes 2}]\right)\mid S_n \right]\right\|^2}{4 n^2 \Delta(t)}  \\
& +  \sum_{k > 2} \frac{\left\|\EE\left[\sum_{i=1}^n (X'_i - X_i)^{\otimes k} 1_{\|X'_i\| \vee \|X_i\| \leq \sqrt{n \Delta(t)}} \mid S_n\right]\right\|^2}{n^{k} k k ! \Delta(t)^{k - 1}}  .
\end{align*}
Now, taking 
\begin{align*}
R(t) \coloneqq & \frac{1}{n} \left\| \sum_{i=1}^n (X_i - X'_i)1_{\|X'_i\| \vee \|X_i\| \geq \sqrt{n \Delta(t)}}\right\|^2 \\
& +  \frac{\left\|\sum_{i=1}^n (X'_i - X_i)^{\otimes 2} 1_{\|X'_i\| \vee \|X_i\| \leq \sqrt{n \Delta(t)}} - 2\EE[X_i^{\otimes 2}] \right\|^2}{4 n^2 \Delta(t)}   \\
& +   \sum_{k > 2} \frac{\left\|\sum_{i=1}^n (X'_i - X_i)^{\otimes k} 1_{\|X'_i\| \vee \|X_i\| \leq \sqrt{n \Delta(t)}}\right\|^2}{n^{k} k k ! \Delta(t)^{k - 1}}
\end{align*}
and applying Jensen's inequality yields 
\[
S(t) \leq \EE[R(t) \mid S_n].
\]
Therefore,
\[
W_2(\nu_n, \gamma) \leq \int_0^\infty e^{-t} \EE[R(t)]^{1/2} dt.
\]
From here, developing the squared terms and using the independence of the $(X_i)_{1 \leq i \leq n}$ and $(X'_i)_{1 \leq i \leq n}$, we obtain 
\[
\EE[R(t)] \leq \sum_{k>0} \frac{1}{n^k k k! \Delta(t)^{k-1}}\left(\sum_{i=1}^n (n-1) I_{k,i}(t) + J_{k,i}(t) \right) ,
\]
where, for any $i \in \{1,\dots, n\}$, 
\begin{align*}
I_{1,i}(t) \coloneqq  & \|\EE[(X_i - X'_i)1_{\|X'_i\| \vee \|X_i\| \geq \sqrt{n \Delta(t)}}]\|^2, \\
I_{2,i}(t) \coloneqq & \|\EE[(X'_i - X_i)^{\otimes 2} 1_{\|X'_i\| \vee \|X_i\| \leq \sqrt{n \Delta(t)}} - 2\EE[X_i^{\otimes 2}]]\|^2, \\
\forall k > 2, I_{k,i}(t) \coloneqq & \|\EE[ (X'_i - X_i)^{\otimes k} 1_{\|X'_i\| \vee \|X_i\| \leq \sqrt{n \Delta(t)}} ]\|^2
\end{align*}
and
\begin{align*}
J_{1,i}(t) \coloneqq  & \EE[\|X_i - X'_i\|^2 1_{\|X'_i\| \vee \|X_i\| \geq \sqrt{n \Delta(t)}}], \\
J_{2,i}(t) \coloneqq & \EE[\| (X'_i - X_i)^{\otimes 2} 1_{\|X'_i\| \vee \|X_i\| \leq \sqrt{n \Delta(t)}}  - 2 \EE[X_i^{\otimes 2}] \|^2], \\
\forall k > 2, J_{k,i}(t) \coloneqq & \EE[\|X'_i - X_i\|^{2k} 1_{\|X'_i\| \vee \|X_i\| \leq \sqrt{n \Delta(t)}}].
\end{align*}
\subsection{Bounding $I_k$}
Let $i \in \{1,\dots,n\}$ and let $k$ be an odd integer. Since $X_i$ and $X'_i$ are i.i.d., 
\[
I_{k,i}(t) = 0.
\]
Let us now deal with $I_{2,i}(t)$. Since $\EE[(X'_i - X_i)^{\otimes 2}] = 2 \EE[X_i^{\otimes 2}]$, we have 
\[
I_{2,i}(t) = \|\EE[(X'_i - X_i)^{\otimes 2} 1_{\|X'_i\| \vee \|X_i\| \geq \sqrt{n \Delta(t)}}]\|^2.
\]
First, since $(X'_i - X_i)^{\otimes 2}$ is positive, 
\[
I_{2,i}(t) \leq 4\|\EE[X_i^{\otimes 2}]\|^2.
\]
Now, taking $l \in (0,m]$, we have 
\begin{align*}
I_{2,i}(t) \leq & \frac{1}{(n \Delta(t))^l} \|\EE[(X'_i - X_i)^{\otimes 2} (\|X_i\| \vee \|X'_i\|)^l]\|^2 \\
& \leq \frac{16}{(n \Delta(t))^l} \|\EE[X_i^{\otimes 2} (\|X_i\| \vee \|X'_i\|)^l]\|^2 \\
& \leq \frac{16}{(n \Delta(t))^l} \|\EE[X_i^{\otimes 2} \|X_i\|^l] + \EE[X_i^{\otimes 2 }] \EE[\|X_i\|^l] \|^2 \\
& \leq \frac{32}{(n \Delta(t))^l}\left(\|\EE[X_i^{\otimes 2} \|X_i\|^l]\|^2 + \|\EE[X_i^{\otimes 2}]\|^2 \EE[\|X_i\|^l]^2 \right)
\end{align*}
and, since this bound is valid for any $l \in (0,m]$, 
\[
I_{2,i}(t) \leq \inf_{l \in (0,m]} \frac{32}{(n \Delta(t))^l}\left(\|\EE[X_i^{\otimes 2} \|X_i\|^l]\|^2 + \|\EE[X_i^{\otimes 2}]\|^2 \EE[\|X_i\|^l]^2 \right).
\]
Similarly, for any even integer $k > 2$ and any $l \in [0,m]$,
\begin{align*}
I_{k,i}(t) & \leq 4^k \|\EE[X_i^{\otimes k}  1_{\|X'_i\| \vee \|X_i\| \leq \sqrt{n \Delta(t)}}]\|^2 \\
& \leq 4^k \|\EE[X_i^{\otimes k} 1_{\|X_i\| \leq \sqrt{n \Delta(t)}}]\|^2 \\
& \leq 4^k \|\EE[X_i^{\otimes 2} \|X_i\|^{k-2} 1_{\|X_i\| \leq \sqrt{n \Delta(t)}}]\|^2 \\
& \leq 4^k ( n \Delta(t))^{k-(l+2)} \|\EE[X_i^{\otimes 2}  \|X_i\|^l]\|^2 \\
& \leq 4^k ( n \Delta(t))^{k-(l+2)} \|\EE[X_i^{\otimes 2} \|X_i\|^l]\|^2,
\end{align*}
leading to
\[
I_{k,i}(t) \leq \inf_{l \in [0,m]} 4^k ( n \Delta(t))^{k-(l+2)} \|\EE[X_i^{\otimes 2} \|X_i\|^l]\|^2
\]
Let us introduce the quantity $M(l)$ defined for $l = 0$ by 
\[
M(0)^2 \coloneqq \frac{1}{n} \sum_{i=1}^n \sum_{k > 0} \frac{16^k}{(2k) (2k)!} \|\EE[X_i^{\otimes 2}]\|^2
\]
and, for any $l \in (0,m]$, by 
\[
M(l)^2 \coloneqq M(0)^2 + \frac{1}{n} \sum_{i=1}^n 4 \|\EE[X_i^{\otimes 2} \|X_i\|^l]\|^2 + 8\|\EE[X_i^{\otimes 2}]\| \EE[\|X_i\|^l].
\]
By combining our bounds on the $I_{k,i}$, we obtain 
\[
\sum_{k > 0} \frac{(n-1) \sum_{i=1}^n I_{k,i}(t)}{n^k k k! (e^{2t}-1)^{k-1}} \leq 
\inf_{l \in [0,m]} \frac{M(l)^2}{n^{l} \Delta(t)^{l+1}}.
\]

\subsection{Bounding $J_k$}
Again, taking $i \in \{1,\dots,n\}$, we have 
\begin{align*}
J_{1,i}(t) & \leq 4 \EE[(\|X_i\| \vee \|X'_i\|)^2 1_{\|X'_i\| \vee \|X_i\| \geq \sqrt{n \Delta(t)}}] \\
& \leq \frac{4}{(n \Delta(t))^m} \EE[(\|X_i\| \vee \|X'_i\|)^{2+m}] \\
& \leq \frac{8}{(n \Delta(t))^m} \EE[\|X_i\|^{2+m}].
\end{align*}
Then,
\begin{align*}
J_{2,i}(t) & \leq 8 \|\EE[X_i^{\otimes 2}]\|^2 + 2 \EE[\|X_i - X'_i\|^4 1_{\|X'_i\| \vee \|X_i\| \leq \sqrt{n \Delta(t)}}] \\
& \leq 8 \|\EE[X_i^{\otimes 2}]\|^2 + 32 \EE[\|X_i\|^4  1_{\|X_i\| \leq \sqrt{n \Delta(t)}}] \\
& \leq 8 \|\EE[X_i^{\otimes 2}]\|^2 + 32 (n \Delta(t))^{1-m/2} \EE[\|X_i\|^{2+m}].
\end{align*}
Finally, for any integer $k > 2$, 
\begin{align*}
J_{k,i}(t) & \leq 4^k \EE[\|X_i\|^{2k} 1_{\|X'_i\| \vee \|X_i\| \leq \sqrt{n \Delta(t)}}]\\
& \leq 4^k (n \Delta(t))^{k - (2 + m)/2} \EE[\|X_i\|^{2+m}].
\end{align*}
Overall, letting
\[
C \coloneqq 8 + \sum_{k > 0} \frac{4^k}{k k!},
\]
we obtained 
\[
\sum_{k > 0} \frac{\sum_{i=1}^n J_{k,i}(t)}{n^k k k! (e^{2t}-1)^{k-1}}  \leq 
\sum_{i=1}^n \frac{C \EE[\|X_i\|^{2+m}]}{\Delta(t)^{m/2} n^{1+m/2}} + \frac{2 \|\EE[X_i^{\otimes 2}]\|^2}{\Delta(t) n^2}.
\]

\subsection{Integration with respect to $t$}
Thanks to the previous computations, we have 
\[
\EE[R(t)] \leq \frac{2 \sum_{i=1}^n \|\EE[X_i^{\otimes 2}]\|^2}{n^2 \Delta(t)}  + \frac{C \sum_{i=1}^n \EE[\|X_i\|^{2+m}] }{n^{1+m/2} \Delta(t)^{m/2}} 
+ \inf_{l \in [0,m]} \frac{M(l)^2}{n^l \Delta(t)^{1+l}}
\]
and thus
\begin{multline*}
\EE[R(t)]^{1/2} \leq \frac{\left(2 \sum_{i=1}^n \|\EE[X_i^{\otimes 2}]\|^2 \right)^{1/2}}{n \Delta(t)^{1/2}}
+ \frac{\left(C \sum_{i=1}^n \EE[\|X_i\|^{2+m}] \right)^{1/2}}{n^{1/2+m/4} \Delta(t)^{m/4}}  \\
+ \inf_{l \in [0,m]} \frac{M(l)}{n^{l/2} \Delta(t)^{(1+l)/2}}.
\end{multline*}
The next step of the proof consists in integrating $e^{-t} \EE[R(t)]^{1/2}$ with respect to $t$. 
First, 
\[
\int_0^\infty \frac{e^{-t}}{\Delta(t)^{1/2}} dt = 1. 
\]
And, since $\int_0^\infty e^{-t} dt = 1$, we have, by Jensen's inequality, 
\[
\int_0^\infty \frac{e^{-t}}{\Delta(t)^{m/4}} dt \leq \left( \int_0^\infty \frac{e^{-t}}{\Delta(t)^{1/2}} dt  \right)^{m/2} = 1.
\]
Let us now deal with the remaining term. Let us first assume that $m \leq 1$. 
Taking $t_0 \geq 0$, we have 
\begin{align*}
\int_0^\infty \inf_{l \in [0,m]} & \frac{e^{-t} M(l) }{n^{l/2} \Delta(t)^{(1+l)/2}} dt \\
&  \leq \int_0^{t_0} \frac{e^{-t_0} M(0)}{\Delta(t)^{1/2}} dt +  \frac{M(m)}{n^{m/2}} \int_{t_0}^\infty \frac{e^{-t}}{(\Delta(t)^{(m+1)/2})}dt \\
& \leq  M(0) (1-e^{-2 t_0})^{1/2} + \frac{M(m)}{n^{m/2}}  \left(\int_{t_0}^\infty \frac{e^{-t}}{e^{2t} - 1} dt \right)^{(m+1)/2} \\
& \leq M(0) \sqrt{2 t_0} + \frac{M(m)}{n^{m/2}}  \left(-e^{-t_0} - \frac{1}{2} \log\left(\frac{1-e^{-t_0}}{1 + e^{-t_0}}\right)\right)^{(m+1)/2} \\
& \leq M(0) \sqrt{2 t_0} - n^{-m/2} M(m) \left(\frac{\log(t_0/2)}{2}\right)^{\frac{m+1}{2}}.
\end{align*}
Since $n \geq 4 > e^2/2$, taking $t_0 = \frac{1}{n} $, we have $ - \frac{\log(t_0/2)}{2} > 1$ and
\[
\int_0^\infty \inf_{l \in [0,m]} \frac{e^{-t} M(l) }{n^{l/2} \Delta(t)^{(1+l)/2}} dt \leq \frac{\sqrt{2} M(0)}{n^{1/2}} + \frac{M(m) \log(n)}{2 n^{m/2}}.
\]
If $ 1 \leq m \leq 2$, performing the same computations with $l = 1$ for $t > t_0$ yields 
\[
\int_0^\infty \inf_{l \in [0,m]} \frac{e^{-t} M(l) }{n^{l/2} \Delta(t)^{(1+l)/2}} dt \leq \frac{\sqrt{2} M(0)}{n^{1/2}} + \frac{M(1) \log(n)}{2 n^{1/2}}.
\]
Finally, if $m = 2$, 
\begin{align*}
\int_0^\infty \inf_{l \in [0,m]} \frac{e^{-t} M(l) }{n^{l/2} \Delta(t)^{(1+l)/2}} dt 
& \leq \left(M(0) \sqrt{2 t_0} + \frac{M(2)}{n}  \int_{t_0}^\infty  \frac{e^{-t}}{\Delta(t)^{3/2}} dt \right) \\
& \leq \left(M(0) \sqrt{2 t_0} + \frac{M(2)}{n} \int_{t_0}^\infty  (2t)^{-3/2} dt \right) \\
& \leq \left(M(0) \sqrt{2 t_0} +  \frac{M(2)}{n \sqrt{2 t_0}} \right).
\end{align*}
Then, taking $t_0 =  \frac{M(2)}{2 M(0) \sqrt{n}}$,
\[
\int_0^\infty  \inf_{l \in [0,m]} \frac{e^{-t} M(l) }{n^{l/2} \Delta(t)^{(1+l)/2}} dt  \leq 2 \sqrt{\frac{M(0) M(2) }{n}},
\]
which concludes the proof of Theorem~\ref{thm:W2CTL}. 

\subsection{Simplifications whenever $\EE[X_i^{\otimes 2}] = I_d$}
\label{sec:iidcase}
Let us now assume that $\EE[X_i^{\otimes 2}] = I_d$ for any $i \in \{1,\dots,n\}$ and let $i \in \{1,\dots,n\}$. We have
\[
\|\EE[X_i^{\otimes 2}]\|^2 \leq d \leq \sqrt{d} \|\EE[X_i^{\otimes 2} \|X_i\|^2]\|.
\] 
Furthermore,
\begin{align*}
\|\EE[X_i^{\otimes 2}]\| \EE[\|X_i\|^2] & = d^{-1/2} \EE[\|X_i\|^2]^{2} \\
& \leq d^{-1/2} \EE[\|X_i\|^4] \\
& \leq d^{-1/2} \sum_{j=1}^d \EE[(X_i)_j^2 \|X_i\|^2] \\
& \leq \left( \sum_{j=1}^d \EE[(X_i)_j^2 \|X_i\|^2]^2 \right)^{1/2} \\
& \leq \|\EE[X_i^{\otimes 2} \|X_i\|^2]\|,
\end{align*}
leading to 
\[
M(0)^2 M(2)^2 \leq \left(\sum_{k >0} \frac{16^k}{2k (2k!)}\right) \left(12 + \sum_{k >0} \frac{16^k}{2k (2k!)}\right)  \frac{d \sum_{i=1}^n \|\EE[X_i^{\otimes 2} \|X_i\|^2]\|^2}{n}.
\]
Similarly, 
\begin{align*}
\EE[\|X_i\|^4] & = \sum_{j=1}^d \EE[(X_i)_j^2 \|X_i\|^2] \\
& \leq  \left( d \sum_{j=1}^d \EE[(X_i)_j^2 \|X_i\|^2]^2\right)^{1/2} \\
& \leq d^{1/2} \|\EE[X_i^{\otimes 2} \|X_i\|^2]\|.
\end{align*}
Therefore, taking 
\[
C' \coloneqq \left(8 + \sum_{k > 0} \frac{4^k}{k (k)!}\right)^{1/2} + \sqrt{2} + 2 \left(\sum_{k >0} \frac{16^k}{2k (2k!)}\right)^{1/4} \left(12 + \sum_{k >0} \frac{16^k}{2k (2k!)} \right)^{1/4},
\]
we have 
\beq
\label{eq:secondmom}
W_2(\nu_n, \gamma) \leq  C' \left( \frac{ d^{1/2} \sum_{i=1}^n \|\EE[X_i^{\otimes 2} \|X_i\|^2]\|}{n} \right)^{1/2}.
\eeq
Finally, remarking that $C' < 14$ and that $\EE[X_i^{\otimes 2}] = I_d$ for all $i$ whenever $(X_i)_{i \in \{1,\dots,n\}}$ are identically distributed concludes the proof of (\ref{eq:W2mainthm}). 

\section{Rates of the multi-dimensional CLT for $W_p$ distances}
\label{sec:CTLcompp}
Let $p > 2, q \in[0,2]$ and $m = \min(2, p+q-2)$. Let $X_1,\dots,X_n$ be independent random variables taking values in $\RR^d$ and such that
\begin{itemize}
\item $\forall i \in \{1,\dots,n\}, \EE[X_i] = 0$; 
\item $\sum_{i=1}^n \EE[X_i^{\otimes 2}] = n I_d$ and
\item $\sum_{i=1}^n \EE[\|X_i\|^{p+q}] < \infty$.
\end{itemize} 
The aim of this Section is to prove the following result. 
\begin{theorem}
\label{thm:WPCTL}
Under the above setting, taking 
\[
\forall l \in [0,2], M(l)^2 = \frac{1}{n} \sum_{i=1}^n \|\EE[X_i^{\otimes 2}]\|^2 \EE[\|X_i\|^l]^2 + \|\EE[X_i^{\otimes 2} \|X_i\|^l]\|^2,
\]
we have that there exists $C_p > 0$ such that
\begin{multline*}
C_p W_2(\nu_n,\gamma) \leq \frac{\left(\sum_{i=1}^n \EE[\|X_i\|^{p+q}]\right)^{1/p}}{n^{1/2+q/2p}} + \frac{\left(\sum_{i=1}^n \EE[\|X_i\|^{2+m}]\right)^{1/2}}{n^{1/2+m/4}}  \\
+ \frac{\left(\sum_{i=1}^n \|\EE[X_i^{\otimes 2}]\|^2\right)^{1/2}}{n} + \begin{cases}
\frac{ M(0)}{n^{1/2}} + \frac{M(m) \log(n)}{n^{m/2}} \text{ if $m < 1$} \\
\frac{M(0)}{n^{1/2}} + \frac{M(1) \log(n)}{n^{1/2}} \text{ if $ 1 \leq m < 2$} \\
\left(\frac{M(0) M(2) }{n}\right)^{1/2} \text{ if $m = 2$}
\end{cases}.
\end{multline*}
\end{theorem}
Taking $((S_n)_t)_{t \geq 0}$ as in the previous question, we have that $((S_n)_0,(S_n)_t)$ and $((S_n)_t,(S_n)_0)$ follow the same law for any $t > 0$. 
Therefore, we can apply Theorem~\ref{thm:WpGaussexch} and perform computations similar to those of the previous Section in order to obtain 
\[
W_p(\nu_n, \gamma) \leq \int_0^\infty e^{-t} \EE[R(t)^{p/2}]^{1/p} dt,
\]
with 
\begin{align*}
R(t) \coloneqq  & \frac{1}{n} \left\| \sum_{i=1}^n (X_i - X'_i)1_{\|X'_i\| \vee \|X_i\| \geq \sqrt{n \Delta(t)}}\right\|^2 \\
& +  \frac{p-1}{\Delta(t)}  \left\|\sum_{i=1}^n \left( \frac{(X'_i - X_i)^{\otimes 2}}{2n} 1_{\|X'_i\| \vee \|X_i\| \leq \sqrt{n \Delta(t)}} \right) - I_d \right\|^2 \\
& +   \sum_{k > 2} \frac{(p-1)^{k-1} \left\|\EE\left[\sum_{i=1}^n (X'_i - X_i)^{\otimes k} 1_{\|X'_i\| \vee \|X_i\| \leq \sqrt{n \Delta(t)}} \mid S_n\right]\right\|^2}{4 n^{k} (k-1) ! \Delta(t)^{k - 1}}.
\end{align*} 
Then, using a multi-dimensional version of Rosenthal inequality such as Theorem 5.2 \citep{pinelis1994}, we obtain that there exists $C_p > 0$ such that 
\[ 
\EE[R(t)^{p/2}] \leq C_p \sum_{k=1}^\infty \frac{(p-1)^{k-1} \left(\left(\sum_{i=1}^n n I_{k,i}(t) + J_{k,i}(t)\right)^{p/2} + \sum_{i=1}^n K_{k,i}(t) \right)}{4 n^k (k-1)! \Delta(t)^{k-1}},
\]
where the $I_{k,i}(t)$ and $J_{k,i}(t)$ are the same as in the previous Section and 
\begin{align*}
K_{1,i}(t) \coloneqq & \EE[\|X_i - X'_i\|^p 1_{\|X'_i\| \vee \|X_i\| \geq \sqrt{n \Delta(t)}}], \\
K_{2,i}(t) \coloneqq & \EE[\|(X'_i - X_i)^{\otimes 2} 1_{\|X'_i\| \vee \|X_i\| \leq \sqrt{n \Delta(t)}} - 2 \EE[X_i^{\otimes 2}]\|^p], \\
\forall k > 2, K_{k,i}(t) \coloneqq & \EE[ \|X'_i - X_i\|^{kp} 1_{\|X'_i\| \vee \|X_i\| \leq \sqrt{n \Delta(t)}} ].
\end{align*}
Then, using arguments similar to the ones used to bound the $J_{k,i}$, 
\begin{align*}
K_{1,i}(t) \leq  & 2^{p+1} (n \Delta(t))^{-q/2} \EE[\|X_i\|^{p+q}], \\
K_{2,i}(t) \leq & 2^{2 p - 1} \|\EE[X_i^{\otimes 2}]\|^p + 2^{3p-1} (n \Delta(t))^{(p-q)/2} \EE[\|X_i\|^{p+q}], \\
\forall k > 2, K_{k,i}(t) \leq & 2^{kp} (n \Delta(t))^{((k-1)p - q)/2} \EE[ \|X_i\|^{p+q}].
\end{align*}
Therefore, there exists $C_p > 0$ such that 
\begin{multline*}
C_p \EE[R(t)^{p/2}]^{1/p} \leq \frac{\left(\sum_{i=1}^n \|\EE[X_i^{\otimes 2}]\|^2 \right)^{1/2}}{n \Delta(t)^{1/2}}  + \frac{\left(\sum_{i=1}^n \EE[\|X_i\|^{2+m}] \right)^{1/2}}{n^{1/2 + m/4} \Delta(t)^{m/4}} \\
+ \frac{\left( \sum_{i=1}^n \|\EE[X_i^{\otimes 2}]\|^p \right)^{1/p}}{n \Delta(t)^{1/2}} + \frac{\left(\sum_{i=1}^n \EE[\|X_i\|^{p+q}] \right)^{1/p}}{n^{1/2 + q/2p} \Delta(t)^{q/(2p)}}  \\
+ \inf_{l \in [0,m]} \frac{\left(\sum_{i=1}^n \|\EE[X_i^{\otimes 2}]\|^2 \EE[\|X_i\|^l]^2 + \|\EE[X_i^{\otimes 2} \|X_i\|^l]\|^2 \right)^{1/2}}{(n \Delta(t))^{(1+l)/2}}  
\end{multline*}
and integrating with respect to $t$ following the arguments of the previous Section concludes the proof of Theorem~\ref{thm:WPCTL} while (\ref{eq:Wpmainthm}) is obtained following the same computations as in Section~\ref{sec:iidcase}. 

\section{Technical results}
\label{sec:technical}
In this Section, we provide the proofs of the intermediary results used to derive Theorems~\ref{thm:mainGauss},\ref{thm:WpGauss} and \ref{thm:WpGaussexch}.
\subsection{Proof of Lemma~\ref{lem:analytic}}
\label{sec:analproof}
Let $\test$ be a bounded and measurable function on $\RR^d$, let $t > 0$ and let $\multiindex$ be a multi-index. By (\ref{eq:ippgauss}), we have 
\[
|\partial^\multiindex P_t \test|^2 \leq \frac{\multiindex!}{(e^{2t}-1)^{|\multiindex|}} P_t \test^2
\]
and, since $\test$ is bounded, there exists $M > 0$ such that 
\beq
\label{eq:derivativebornee}
|\partial^\multiindex P_t \test|^2 \leq \frac{M \multiindex !}{(e^{2t}-1)^{|\multiindex|}}. 
\eeq
Then, since $\|X_t - X_0\|$ is bounded as well, we have that there exists $C> 0$  such that 
\beq 
\label{eq:derivativesumbornee}
\sum_{|\multiindex| >0} \left| \frac{(X_t-X_0)^\multiindex}{\multiindex !} \partial^\multiindex P_t \test(X_0) \right| \leq \sum_{|\multiindex| >0} \frac{C^{|\multiindex|}}{\sqrt{\multiindex !}} < \infty. 
\eeq
almost surely. Therefore
\[
\EE\left[\sum_{|\multiindex| >0} \left| \frac{(X_t-X_0)^\multiindex}{\multiindex !} \partial^\multiindex P_t \test(X_0) \right| \right] < \infty
\]
and 
\[
\sum_{|\multiindex| >0} \EE\left[\frac{(X_t-X_0)^\multiindex}{\multiindex !} \partial^\multiindex P_t \test(X_0)\mid X_0 \right]
\]
exists.

Now, using a Taylor expansion with remainder, we obtain that there exists $\xi$ on the segment $[X_0,X_t]$ such that
\begin{multline*}
P_t \test(X_t) - P_t \test(X_0) = \\
\sum_{0<|\multiindex|<l} \frac{(X_t-X_0)^{\multiindex}}{\multiindex !} \partial^\multiindex P_t \test(X_0) + \sum_{|\multiindex|=l} \frac{(X_t-X_0)^{\multiindex}}{\multiindex!}  \partial^\multiindex P_t \test (\xi). 
\end{multline*}
From here, we have  
\begin{multline*}
\left| P_t \test(X_t) - P_t \test(X_0) - \sum_{0<|\multiindex|<l} \frac{1}{\multiindex!} (X_t-X_0)^{\multiindex} \partial^\multiindex P_t\test  (X_0) \right| \leq \\
 \sum_{|\multiindex| = l} \frac{|X_t-X_0|^\multiindex}{\multiindex!} |\partial^\multiindex P_t \test(\xi)|.
\end{multline*}
Then, by (\ref{eq:derivativesumbornee}),
\[
\EE[P_t \test(X_t) - P_t \test(X_0)]  = \EE\left[\sum_{|\multiindex| >0} \frac{(X_t-X_0)^\multiindex}{\multiindex !} \partial^\multiindex P_t \test(X_0) \right].
\]
and
\[
\EE[P_t \test(X_t) - P_t \test(X_0)] = \sum_{|\multiindex|> 0} \EE\left[ \frac{1}{\multiindex !} \EE[(X_t-X_0)^{\multiindex} \mid X_0] \partial^\multiindex P_t \test(X_0)\right].
\]
\subsection{Proof of Lemma~\ref{lem:hypercon}}
\label{sec:hypercon}
Let $(M_\multiindex)_{\multiindex \in \NN^d}$ such that $M_\multiindex \in \RR^d$ for any multi-index $\multiindex$ and let $Z$ be a Gaussian random variable. 
Let us start with the case $1 \leq p < 2$. By Jensen's inequality,
\begin{align*}
\EE\left[\left\| \sum_{\multiindex} M_\multiindex H_\multiindex(Z) \right\|^{p}\right]^{2/p} & \leq \EE\left[\left\|\sum_{\multiindex} M_\multiindex H_\multiindex(Z) \right\|^2\right] \\
 & \leq \EE\left[\sum_{\multiindex, \multiindex'} M_\multiindex \cdot M_{\multiindex'} H_\multiindex(Z) H_{\multiindex'}(Z)\right].
\end{align*}
Then, since $\EE[H_\multiindex(Z) H_{\multiindex'}(Z)] = 0$ for any two different multi-indices $\multiindex, \multiindex'$, 
\[
\EE\left[\left\| \sum_{\multiindex} M_\multiindex H_\multiindex(Z) \right\|^{p}\right]^{2/p} \leq \sum_{\multiindex} \multiindex ! \|M_\multiindex\|^2.
\]

Now, let $p > 2$ and $t \coloneqq \log(\sqrt{p-1})$. 
Since the Ornstein-Uhlenbeck semigroup $(P_t)_{t \geq 0}$ is hypercontractive (see e.g. Theorem 5.2.3 \citep{Markov}), we have 
\[
\forall \test \in L^2(\gamma), \EE[|P_t \test(Z)|^p]^{1/p} \leq \EE[\test(Z)^2]^{1/2}. 
\]
This inequality can be readily extended to vector-valued functions $\test$, in which case we have
\[
\forall \test, \|\test\|\in L^2(\gamma), \EE[\|P_t \test(Z)\|^p]^{1/p} \leq \EE[(P_t \|\test(Z)\|)^p]^{1/p} \leq \EE[\|\test(Z)\|^2]^{1/2}.
\]
For any multi-index $\multiindex$, the Hermite polynomial $H_\multiindex$ is an eigenvector of $P_t$ with eigenvalue $e^{-  |\multiindex| t} = (p-1)^{-|\multiindex|/2}$. 
Therefore, 
\begin{align*}
\EE\left[\left\|\sum_{\multiindex} M_{\multiindex} H_\multiindex(Z)\right\|^{p}\right]^{2/p} & = \EE\left[\left\|\sum_{\multiindex} (p-1)^{|\multiindex|/2} M_\multiindex P_t H_{\multiindex}(Z)\right\|^p\right]^{2/p} \\
& =  \EE\left[\left\|P_t \sum_{\multiindex} (p-1)^{|\multiindex|/2} M_\multiindex H_{\multiindex}(Z)\right\|^p\right]^{2/p} \\
& \leq \EE\left[\left\|\sum_{\multiindex} (p-1)^{|\multiindex|/2} M_\multiindex H_{\multiindex}(Z)\right\|^2\right] \\
& \leq \sum_{\multiindex}(p-1)^{|\multiindex|} \multiindex !  \|M_\multiindex\|^2,
\end{align*}
concluding the proof.
\section{Approximation arguments}
\label{sec:W2approx}
In this Section, we present the approximation arguments necessary to conclude the proof of Theorem \ref{thm:WpGaussexch}. Similar arguments can be used to obtain 
Theorems~\ref{thm:mainGauss} and \ref{thm:WpGauss}.

Suppose the measure $\nu$ and the stochastic process $(X_t)_{t \geq 0}$ satisfy the assumptions of Theorem~\ref{thm:WpGaussexch}.  
Let $s > 0$ and
\begin{align*} 
S_p(t) & \coloneqq \left\|\EE\left[\frac{X_t-X_0}{s}+X_0 \mid X_0\right]\right\|^2 \\
& +  \frac{ \max(1, p-1)}{e^{2t}-1}  \left\|\EE\left[\frac{(X_t-X_0)^{\otimes 2}}{2 s} - I_d \mid X_0 \right]\right\|^2 \\
& + \sum_{k > 2} \frac{\max(1, p-1)^{k - 1} }{4 s^2 (k-1) ! (e^{2t}-1)^{k - 1}}  \|\EE[(X_t-X_0)^{\otimes k} \mid X_0]\|^2.
\end{align*}

Let $R > 1, \epsilon_1 = R^{-1}$ and $0<\epsilon_2<1$ . For any $t > 0$, let $X_t^R$ be the orthogonal projection of $X_t$ on $\mathcal{B}(0,R)$, the ball of radius $R$ centered at $0$. Let $Z$ be a standard normal random variable, $N$ be a random variable with smooth density and taking values in the ball of radius $1$ and let $I$ be a Bernoulli random variable with parameter $\epsilon_1$ such that $(X_t)_{t \geq 0}, Z, N$ and $I$ are independent. Finally, let $ U = \epsilon_1 N$. 
For any $t > 0$, let
\[
\tilde{X}_t \coloneqq I Z + (1-I) ( U +  X^R_t 1_{\epsilon_2 \leq t \leq \epsilon_2^{-1}} + X^R_0 (1_{\epsilon_2 > t} + 1_{ t > \epsilon_2^{-1}})). 
\]
Let $\tilde{\nu}_R$ be the law of $\tilde{X}_0$. This measure admits a density $h$ with respect to the measure $\gamma$ such that $h = \epsilon_1 + f$ with $f \in \mathcal{C}^\infty_c$. Furthermore, for any $t > 0$,  $(\tilde{X}_0,\tilde{X}_t)$ and  $(\tilde{X}_t,\tilde{X}_0)$ follow the same law. Therefore, we can follow the computations of Section~\ref{sec:WpGaussexch} and use the triangle inequality to obtain 
\begin{multline}
\label{eq:approxGaussmain}
W_p(\tilde{\nu}_R,\gamma) \leq \epsilon_1 \int_0^\infty e^{-t} \EE[S_Z(t)^{p/2}]^{1/p} dt + (\int_0^{\epsilon_2} + \int_{1/\epsilon_2}^\infty) e^{-t} \EE[\tilde{S}_{p,1}(t)^{p/2}]^{1/p} dt  \\
+ \int_{\epsilon_2}^{1/\epsilon_2}   e^{-t} \EE[\tilde{S}_{p,2}(t)^{p/2}]^{1/p} dt,
\end{multline}
where
\begin{align*}
S_Z(t) \coloneqq & \|Z\|^2 + \frac{  d\max(1, p-1) }{e^{2t} -1}, \\
\tilde{S}_{p,1}(t) \coloneqq & \|X_0^R + U\|^2 +  \frac{d\max(1, p-1)}{e^{2t} -1}
\end{align*}
and
\begin{align*}
\tilde{S}_{p,2}(t) \coloneqq  & \left\|\EE\left[\frac{X^R_t-X^R_0}{s} + (X^R_0 + U) \mid X^R_0 + U\right]\right\|^2 \\
& + \frac{\max(1, p-1)}{e^{2t} -1}  \left\|\EE\left[\frac{(X^R_t-X^R_0)^{\otimes 2}}{2 s}  - I_d \mid X^R_0 + U \right]\right\|^2 \\
& +  \sum_{k > 2} \frac{\max(1, p-1)^{k - 1}}{4 s^2 (k-1) ! (e^{2t}-1)^{k - 1}} \|\EE[(X^R_t-X^R_0)^{\otimes k} \mid X^R_0 + U]\|^2.
\end{align*}
First, since $Z$ admits a finite moment of order $p$, there exists $C > 0$ such that 
\beq
\label{eq:SZ}
\int_0^\infty e^{-t} \EE[S_Z(t)^{p/2}]^{1/p} dt \leq C. 
\eeq
Then, since $X_0^R$ is the orthogonal projection of $X_0$ on $\mathcal{B}(0,R)$, 
\[
\|X_0^R + U\| \leq \|X_0^R \| + \epsilon_1 \leq \|X_0\| + 1
\]
and, since $\nu$ admits a finite moment of order $p$, there exists $C > 0$ such that 
\[
\EE[\tilde{S}_{p,1}(t)^{p/2}]^{1/p} \leq C\left( 1 + \frac{1}{\sqrt{e^{2t} -1} }\right). 
\]
Therefore, there exists $C > 0$ such that 
\beq
\label{eq:approxGaussaux1}
(\int_0^{\epsilon_2} + \int_{1/\epsilon_2}^\infty) e^{-t} \EE[\tilde{S}_{p,1}(t)^{p/2}]^{1/p} dt \leq C (\sqrt{\epsilon_2} + e^{-\epsilon_2}).
\eeq
Now, let 
\begin{align*}
\tilde{S}_{p,3}(t) \coloneqq & \left\|\EE\left[\frac{X^R_t-X^R_0}{s} + X^R_0  \mid X^R_0 + U \right]\right\|^2 \\
& + \frac{\max(1, p-1)}{e^{2t} -1}  \left\|\EE\left[\frac{(X^R_t-X^R_0)^{\otimes 2}}{2 s}  - I_d \mid X^R_0 + U \right]\right\|^2 \\
& +  \sum_{k > 2} \frac{\max(1, p-1)^{k - 1} }{4 s^2 (k-1) ! (e^{2t}-1)^{k - 1}} \|\EE[(X^R_t-X^R_0)^{\otimes k} \mid X^R_0 + U]\|^2.
\end{align*}
By the triangle inequality, we have that
\[
\EE[\tilde{S}_{p,2}^{p/2}]^{1/p} \leq \EE[\tilde{S}_{p,3}^{p/2}]^{1/p} + \EE[\|U\|^{p}]^{1/p}
\]
and, since $\|U\| \leq \epsilon_1$, 
\beq
\label{eq:approxGaussaux0}
\EE[\tilde{S}_{p,2}^{p/2}]^{1/p} - \EE[\tilde{S}_{p,3}^{p/2}]^{1/p} \leq \epsilon_1. 
\eeq
Finally, let 
\begin{align*}
\tilde{S}_{p,4}(t) \coloneqq & \left\|\EE\left[\frac{X^R_t-X^R_0}{s} + X^R_0  \mid X_0 \right]\right\|^2 \\
& + \frac{\max(1, p-1)}{e^{2t} -1}  \left\|\EE\left[\frac{(X^R_t-X^R_0)^{\otimes 2}}{2 s}  - I_d \mid X_0  \right]\right\|^2 \\
& +  \sum_{k > 2} \frac{\max(1, p-1)^{k - 1} }{4 s^2 (k-1) ! (e^{2t}-1)^{k - 1}} \|\EE[(X^R_t-X^R_0)^{\otimes k} \mid X_0]\|^2.
\end{align*}
Since $(X^R_t)_{t \geq 0}$ and $U$ are independent and since $X^R_0$ is $X_0$-measurable, we have
\beq
\label{eq:approxGaussaux2}
\EE[\tilde{S}_{p,3}(t)^{p/2}]^{1/p} \leq \EE[\tilde{S}_{p,4}(t) ^{p/2}]^{1/p}.
\eeq
From here, 
\begin{align*}
&\tilde{S}_{p,4}(t)  - S_p(t) - \left(\|X_0^R\|^2 + \frac{d \max(1, p-1)}{e^{2t} - 1} \right)  1_{X_0 \notin \mathcal{B}(0,R)} \\
&  \leq  \sum_{k > 0} \frac{\max(1, p-1)^{k-1} \EE[\|X^R_t - X_0^R\|^{2k}}{4 s^2 (k-1)!(e^{2t}-1)^{k-1}}  (1_{X_t \notin \mathcal{B}(0,R)} + 1_{X_0 \notin \mathcal{B}(0,R)}) \mid X_0] \\
& \leq \EE\left[\sum_{k >0} \frac{\max(1, p-1)^{k-1}\|X^R_t - X_0^R\|^{2k} }{4 s^2 (k-1)! (e^{2t}-1)^{k-1}} (1_{X_t \notin \mathcal{B}(0,R)} + 1_{X_0 \notin \mathcal{B}(0,R)}) \mid X_0 \right] \\
& \leq \frac{1}{4 s^2} \EE\left[\|X^R_t - X^R_0\|^2 e^{\frac{\max(1, p-1) \|X^R_t - X^R_0\|^2}{e^{2t} -1}}  (1_{X_t \notin \mathcal{B}(0,R)} + 1_{X_0 \notin \mathcal{B}(0,R)}) \mid X_0 \right]. 
\end{align*}
Thus, applying the triangle inequality Jensen's inequality yields  
\begin{align*}
\EE[&\tilde{S}_{p,4} (t)^{p/2}]^{1/p} \leq  \EE[S_{p}(t)^{p/2}]^{1/p} + \EE[\|X_0^R\|^{p} 1_{X_0 \notin \mathcal{B}(0,R)}]^{1/p} \\
&+ \left(P(X_0 \notin \mathcal{B}(0,R))\frac{d \max(1, p-1)}{e^{2t} - 1}\right)^{1/2} \\& 
+ \frac{1}{2 s} \EE\left[\|X^R_t - X^R_0\|^p e^{\frac{p \max(1, p-1) \|X^R_t - X^R_0\|^2}{2(e^{2t} -1)}}  (1_{X_t \notin \mathcal{B}(0,R)} + 1_{X_0 \notin \mathcal{B}(0,R)})^p \right]^{1/p}.
\end{align*}
Since $X^R_t$ is the orthogonal projection of $X_t$ on the convex set $\mathcal{B}(0,R)$, we have $\|X_0^R\| \leq \|X_0\|$ and $\|X^R_t - X^R_0\| \leq \|X_t - X_0\|$. Hence, 
\begin{align*}
\EE[\tilde{S}_{p,4}& (t)^{p/2}]^{1/p} \leq \EE[S_{p}(t)^{p/2}]^{1/p} + \EE[\|X_0\|^{p} 1_{X_0 \notin \mathcal{B}(0,R)}]^{1/p}  \\ 
&+ \left(P(X_0 \notin \mathcal{B}(0,R))\frac{d \max(1, p-1)}{e^{2t} - 1}\right)^{1/2}  \\
&+ \frac{1}{2 s} \EE\left[\|X_t - X_0\|^p e^{\frac{p \max(1, p-1) \|X_t - X_0\|^2}{2(e^{2t} -1)}}  (1_{X_t \notin \mathcal{B}(0,R)} + 1_{X_0 \notin \mathcal{B}(0,R)})\right]^{1/p}.
\end{align*}
By (\ref{eq:condition3}), there exists $\xi,M > 0$, depending on $\epsilon_2$, such that, for any $t \in [\epsilon_2, \epsilon_2^{-1}]$, 
\[
\EE\left[\|X_t - X_0\|^{p(1+\xi)} e^{\frac{ (1+\xi) p \max(1, p-1) \|X_t - X_0\|^2}{2(e^{2t}-1)}}\right] \leq M. 
\]
Hence, using H\"older's inequality, we obtain that there exists $M'(\epsilon_2), C(\epsilon_2) > 0$ such that 
\begin{multline*}
\EE[\tilde{S}_{p,4}(t)^{p/2}]^{1/p} \leq \EE[S_{p}(t)^{p/2}]^{1/p} + \EE[\|X_0\|^{p} 1_{X_0 \notin \mathcal{B}(0,R)}]^{1/p} \\
 + \left(P(X_0 \notin \mathcal{B}(0,R))\frac{d \max(1, p-1)}{e^{2t} - 1}\right)^{1/2}  
 + M'(\epsilon_2) P(X_0 \notin \mathcal{B}(0,R))^{C(\epsilon_2)}.
\end{multline*}
Combining this bound with (\ref{eq:approxGaussmain}), (\ref{eq:SZ}), (\ref{eq:approxGaussaux1}), (\ref{eq:approxGaussaux0}) and (\ref{eq:approxGaussaux2}), we obtain that there exists $C > 0$ and $C_1(\epsilon_2),C_2(\epsilon_2) > 0$ such that 
\begin{multline*}
W_p(\tilde{\nu}_R,\gamma) \leq \int_0^\infty \EE[S_p(t)^{p/2}]^{1/p} dt +  C_1(\epsilon_2)P(X_0 \notin \mathcal{B}(0,R))^{C_2(\epsilon_2)} \\
+ C(\sqrt{\epsilon_2} + e^{-1/\epsilon_2} + \epsilon_1 +  \EE[\|X_0\|^{p} 1_{X_0 \notin \mathcal{B}(0,R)}]^{1/p}).
\end{multline*}
Since $X_0$ has a finite moment of order $p$ and since $\epsilon_1 = R^{-1}$, letting $R$ go to infinity and $\epsilon_2$ go to zero yields 
\[
\lim_{R \rightarrow \infty} W_p(\tilde{\nu}_R,\gamma) \leq \int_0^\infty \EE[S_{p}(t)^{p/2}]^{1/p} dt.
\]
On the other hand, 
when $R$ goes to infinity, we have that $\tilde{\nu}_R$ converge weakly to $\nu$ and the $p$-moment of $\tilde{\nu}_R$ converges to the $p$-moment of $\nu$. Thus, by Theorem~6.9 \citep{Villani}, $W_p(\tilde{\nu}_R, \nu)$ converges to zero as $R$ goes to infinity. Therefore, 
\[
W_p(\nu,\gamma) \leq \lim_{R \rightarrow \infty} (W_p(\tilde{\nu}_R,\gamma) + W_p(\tilde{\nu}_R, \nu)) \leq \int_0^\infty \EE[S_{p}(t)^{p/2}]^{1/p} dt,
\]
concluding the proof of Theorem \ref{thm:WpGaussexch}.  

\section*{Acknowledgements} 
The author would like to thank Michel Ledoux for his many comments and advice regarding the redaction of this paper as well as J\'er\^ome Dedecker, Yvik Swan, Fr\'ed\'eric Chazal and anonymous reviewers for their multiple remarks. 

\bibliography{Bibliography}

\begin{thebibliography}{10}
\providecommand{\url}[1]{{#1}}
\providecommand{\urlprefix}{URL }
\expandafter\ifx\csname urlstyle\endcsname\relax
  \providecommand{\doi}[1]{DOI~\discretionary{}{}{}#1}\else
  \providecommand{\doi}{DOI~\discretionary{}{}{}\begingroup
  \urlstyle{rm}\Url}\fi

\bibitem{Markov}
Bakry, D., Gentil, I., Ledoux, M.: {Analysis and Geometry of Markov Diffusion
  operators}.
\newblock Grundlehren der mathematischen Wissenschaften, Vol. 348. {Springer}
  (2014)

\bibitem{Bobkov}
Bobkov, S.G.: Entropic approach to e. rio’s central limit theorem for w2
  transport distance.
\newblock Statistics and Probability Letters \textbf{83}(7), 1644--1648 (2013)

\bibitem{Fathi}
{Courtade}, T.A., {Fathi}, M., {Pananjady}, A.: {Existence of Stein Kernels
  under a Spectral Gap, and Discrepancy Bound}.
\newblock ArXiv e-prints  (2017)

\bibitem{Eldan}
{Eldan}, R., {Mikulincer}, D., {Zhai}, A.: {The CLT in high dimensions:
  quantitative bounds via martingale embedding}.
\newblock ArXiv e-prints  (2018)

\bibitem{Fathi2}
{Fathi}, M.: {Stein kernels and moment maps}.
\newblock ArXiv e-prints  (2018)

\bibitem{zerobias}
Goldstein, L., Reinert, G.: Stein's method and the zero bias transformation
  with application to simple random sampling.
\newblock Ann. Appl. Probab. \textbf{7}(4), 935--952 (1997)

\bibitem{sizebias}
Goldstein, L., Rinott, Y.: Multivariate normal approximations by stein's method
  and size bias couplings.
\newblock Journal of Applied Probability \textbf{33}, 1--17 (1996)

\bibitem{Stein}
Ledoux, M., Nourdin, I., Peccati, G.: Stein's method, logarithmic sobolev and
  transport inequalities.
\newblock Geometric and Functional Analysis \textbf{25}(1), 256--306 (2015)

\bibitem{OV}
Otto, F., Villani, C.: Generalization of an inequality by talagrand and links
  with the logarithmic sobolev inequality.
\newblock Journal of Functional Analysis \textbf{173}(2), 361 -- 400 (2000)

\bibitem{pinelis1994}
Pinelis, I.: Optimum bounds for the distributions of martingales in banach
  spaces.
\newblock Ann. Probab. \textbf{22}(4), 1679--1706 (1994)

\bibitem{reinert2009}
Reinert, G., Röllin, A.: Multivariate normal approximation with stein’s
  method of exchangeable pairs under a general linearity condition.
\newblock Ann. Probab. \textbf{37}(6), 2150--2173 (2009)

\bibitem{rio2009}
Rio, E.: Upper bounds for minimal distances in the central limit theorem.
\newblock Ann. Inst. H. Poincaré Probab. Statist. \textbf{45}(3), 802--817
  (2009)

\bibitem{Rollin}
R{\"o}llin, A.: A note on the exchangeability condition in stein’s method.
\newblock Statistics and Probability Letters \textbf{78}(13), 1800 -- 1806
  (2008)

\bibitem{OriginalStein}
Stein, C.: A bound for the error in the normal approximation to the
  distribution of a sum of dependent random variables.
\newblock In: Proceedings of the Sixth Berkeley Symposium on Mathematical
  Statistics and Probability, Volume 2: Probability Theory, pp. 583--602.
  University of California Press, Berkeley, Calif. (1972)

\bibitem{Villani}
Villani, C.: Optimal transport : old and new.
\newblock Grundlehren der mathematischen Wissenschaften. Springer, Berlin
  (2009)

\bibitem{WangOV}
Wang, F.Y.: Probability distance inequalities on riemannian manifolds and path
  spaces.
\newblock Journal of Functional Analysis \textbf{206}(1), 167 -- 190 (2004)

\bibitem{Zhai}
Zhai, A.: A high-dimensional clt in w2 distance with near optimal convergence
  rate.
\newblock Probability Theory and Related Fields pp. 1--25 (2017)

\end{thebibliography}

\end{document}